\documentclass[10pt,reqno]{amsart}

\usepackage[dvipsnames]{xcolor}
\usepackage{amsmath,amssymb,amsthm,amsfonts,enumerate,tikz,bm}
\usepackage{mathrsfs}
\usetikzlibrary{matrix,arrows,positioning,calc}
\usepackage[all]{xy}
\usepackage[bookmarksnumbered,colorlinks]{hyperref}
\usepackage{url}
\numberwithin{equation}{section}
\usepackage{graphicx}
\usepackage[T1]{fontenc}
\usepackage{textcomp}
\usepackage{palatino,helvet}

\usepackage{footmisc}
\usepackage{float}
\usepackage{caption}

\newtheorem{definition}{Definition}[section]
\newtheorem{remark}[definition]{Remark}
\newtheorem{example}[definition]{Example}
\newtheorem{theorem}[definition]{Theorem}
\newtheorem{proposition}[definition]{Proposition}
\newtheorem{lemma}[definition]{Lemma}
\newtheorem{corollary}[definition]{Corollary}
\theoremstyle{remark}
\numberwithin{equation}{section}

\usepackage{paralist}

\usepackage{scrextend}
\deffootnote{2em}{0em}{\thefootnotemark\quad}

\newcommand{\mbZ}{\mathbb{Z}}
\newcommand{\mcA}{\mathcal{A}}
\newcommand{\mcB}{\mathcal{B}}
\newcommand{\mcC}{\mathcal{C}}
\newcommand{\mcD}{\mathcal{D}}

\newcommand{\mcF}{\mathcal{F}}

\newcommand{\mcI}{\mathcal{I}}
\newcommand{\mcL}{\mathcal{L}}
\newcommand{\mcP}{\mathcal{P}}

\newcommand{\mcX}{\mathcal{X}}
\newcommand{\mcY}{\mathcal{Y}}
\newcommand{\mcGP}{\mathcal{GP}}

\newcommand{\mcGI}{\mathcal{GI}}

\newcommand{\modu}{\mathrm{mod}}
\newcommand{\add}{\mathrm{add}}

\newcommand{\Mod}{\mathsf{Mod}}
\newcommand{\fgMod}{\mathsf{mod}}
\newcommand{\Ch}{\mathsf{Ch}}

\newcommand{\resdim}{{\rm resdim}}
\newcommand{\coresdim}{{\rm coresdim}}
\newcommand{\pd}{{\rm pd}}
\newcommand{\id}{{\rm id}}
\newcommand{\Gpd}{{\rm Gpd}}
\newcommand{\Gid}{{\rm Gid}}

\newcommand{\Hom}{{\rm Hom}}
\newcommand{\Ext}{{\rm Ext}}

\newcommand{\Ker}{\text{Ker}}

\makeatletter
\def\@seccntformat#1{%
  \protect\textup{\protect\@secnumfont
    \ifnum\pdfstrcmp{section}{#1}=0 \scshape\bfseries\fi
    \ifnum\pdfstrcmp{subsection}{#1}=0 \bfseries\fi
    \csname the#1\endcsname
    \protect\@secnumpunct
  }%
}
\makeatother

\begin{document}

\title{$m$-periodic Gorenstein objects}
\thanks{2020 MSC: 18G25 (18G10; 18G20; 16E65)}
\thanks{Key Words: relative Gorenstein projective (injective) objects, strongly Gorenstein projective modules, GP-admissible pairs, cluster tilting subcategories.}

\author{Mindy Y. Huerta}
\address[M. Y. Huerta]{Facultad de Ciencias, Universidad Nacional Aut\'onoma de M\'exico. Circuito Exterior, Ciudad Universitaria. CP04510. Mexico City, MEXICO. 
}
\email{mindyhp90@ciencias.unam.mx}

\author{Octavio Mendoza}
\address[O. Mendoza]{Instituto de Matem\'aticas. Universidad Nacional Aut\'onoma de M\'exico. Circuito Exterior, Ciudad Universitaria. CP04510. Mexico City, MEXICO}
\email{omendoza@matem.unam.mx}

\author{Marco A. P\'erez}
\address[M. A. P\'erez]{Instituto de Matem\'atica y Estad\'istica ``Prof. Ing. Rafael Laguardia''. Facultad de Ingenier\'ia. Universidad de la Rep\'ublica. CP11300. Montevideo, URUGUAY}
\email{mperez@fing.edu.uy}

\maketitle

\begin{abstract}
We present and study the concept of $m$-periodic Gorenstein objects relative to a pair $(\mathcal{A,B})$ of classes of objects in an abelian category, as a generalization of $m$-strongly Gorenstein projective modules over associative rings. We prove several properties when $(\mathcal{A,B})$ satisfies certain homological conditions, like for instance when $(\mcA,\mcB)$ is a GP-admissible pair. Connections to Gorenstein objects and Gorenstein homological dimensions relative to these pairs are also established.  
\end{abstract}


\pagestyle{myheadings}
\markboth{\rightline {\scriptsize M. Huerta, O. Mendoza and M. A. P\'{e}rez}}
         {\leftline{\scriptsize $m$-periodic Gorenstein objects}}


\section*{Introduction}

Let $R$ be an associative ring with identity. In \cite{BM2}, Bennis and Mahdou introduced the notion of strongly Gorenstein projective modules over $R$. These are particular cases of Gorenstein projective $R$-modules $M$ which are part of a short exact sequence $M \rightarrowtail P \twoheadrightarrow M$, with $P$ projective. In other words, strongly Gorenstein projective modules are Gorenstein projective modules which are projective-periodic, in the sense of Bazzoni, Cort\'es-Izurdiaga and Estrada \cite{BCE}. These modules have turned out to be useful in the computation of global Gorenstein dimensions. Indeed, Bennis and Mahdou proved in \cite{BMglobal} that 
\[
{\rm sup}\{ {\rm Gpd}(M) {\rm \ : \ } M \in \Mod(R) \} = {\rm sup}\{ {\rm Gid}(M) {\rm \ : \ } M \in \Mod(R) \},
\]
where $\Mod(R)$ denotes the category of left $R$-modules, and ${\rm Gpd}(M)$ and ${\rm Gid}(M)$ denote the Gorenstein projective and Gorenstein injective dimensions of $M$. Part of the proof of this equality is based in the characterization of Gorenstein projective modules as direct summands of strongly Gorenstein projective modules. 

A further generalization of strongly Gorenstein projective modules was later proposed in \cite{BM}, also by Bennis and Mahdou, which they called $m$-strongly Gorenstein projective modules. These are defined as those $M \in \Mod(R)$ for which there exists an exact sequence $M \rightarrowtail P_m \to \cdots \to P_1 \twoheadrightarrow M$, with $P_k$ projective for every $1 \leq k \leq m$, that remains exact after applying the contravariant functor $\Hom(-,Q)$, with $Q$ running over the class of projective modules. 

In the present article, we propose in the setting of an abelian category $\mcC$ a relativization of $m$-strongly Gorenstein modules, which we call $m$-periodic $(\mathcal{A,B})$-Gorenstein projective objects, where $\mcA$ and $\mcB$ are classes of objects in $\mcC$. This follows the relativization of Gorenstein projective modules proposed in \cite{BMS} by Becerril, Santiago and the second author, which are known as $(\mathcal{A,B})$-Gorenstein projective objects. These are defined as cycles of an exact complex $A_\bullet$ of objects in $\mcA$ such that $\Hom(A_\bullet, B)$ is exact for every $B \in \mcB$. So it is natural to think of a notion of $m$-strongly Gorenstein projective objects relative to $(\mathcal{A,B})$ in the following way: $(\mcA,\mcB)$-Gorenstein projective objects $M$ which are $\mcA$-periodic with period $m$, that is, for which there is an exact sequence of the form
\begin{align}
M \rightarrowtail A_m \to \cdots \to A_1 \twoheadrightarrow M \label{eqn:loop_intro}
\end{align}
with $A_k \in \mcA$ for every $1 \leq k \leq m$. Several well known results for strongly and $m$-strongly Gorenstein projective modules will be also true for $m$-periodic $(\mcA,\mcB)$-Gorenstein projective objects, provided that certain conditions are fulfilled by $\mcA$ and $\mcB$. The assumptions we consider on $\mcA$ and $\mcB$ determine the organization of this article, which we describe below.


\subsection*{Organization of the paper}

Section \ref{sec:preliminaries} is devoted to recall some preliminaries from homological algebra, and to fix most of the notation we shall be using throughout. Then in Section \ref{sec:acyclicity} we prove some preliminary results concerning acyclicity conditions of complexes relative to Hom functors. Most of these results are presented under the terminology of $\mcA$-loops of length $m$ at objects in an abelian category. These are chain complexes that result from glueing at an object $M$ infinitely many complexes of the form \eqref{eqn:loop_intro} (see Definition \ref{def:loop}). The most remarkable result in this section is Corollary \ref{cor:Li^perp=M^perp}, where we give a characterization of the condition $M \in {}^\perp\mcB$ for $\mcA$-loops at $M$, where $\mcB$ is a class of objects $B$ satisfying $\Ext^i(A,B) = 0$ for every $A \in \mcA$ and $i \geq 1$. This characterization will be key in providing an alternative description of the acyclicity condition that defines $m$-periodic $(\mcA,\mcB)$-Gorenstein projective objects (compare Definitions \ref{def:sGP} and \ref{def:wsGP}). We also show some characterizations of acyclicity for $\mcA$-loops in the case where $\mcA$ is an $(n+1)$-rigid subcategory (see Proposition \ref{prop:ex}).

In Section \ref{sec: m-period} we define the class of $m$-periodic $(\mcA,\mcB)$-Gorenstein projective objects in abelian categories. These are objects of the form $Z_{mk}(A_\bullet)$, with $k \in \mathbb{Z}$, where $A_\bullet$ is an $\mcA$-loop of length $m$ which is also $\Hom(-,\mcB)$-acyclic. If the later condition is replaced by saying that $A_\bullet$ has cycles in ${}^{\perp}\mcB$, we obtain a particular family of $m$-periodic $(\mcA,\mcB)$-Gorenstein objects which we call \emph{weak}. Several examples that support these definitions are displayed in Example \ref{ex:sGP}. In the particular instance where we consider $m$-periodic Gorenstein projective modules relative to the pair (finitely presented modules, injective modules), we obtain a characterization of Noetherian rings in terms of these relative periodic Gorenstein modules (see Proposition \ref{prop:characterization_noetherian}). 

Several good properties and characterizations of $m$-periodic $(\mcA,\mcB)$-Gorenstein objects occur when the hereditary condition $\pd_{\mcB}(\mcA) = 0$ (that is, every object in $\mcA$ is projective relative to every object in $\mcB$) is assumed. These are studied in Section \ref{sec: hereditary}. Our first result in Proposition \ref{pro:WSGP=SGP} is that the condition $\pd_{\mcB}(\mcA) = 0$ is equivalent to the fact that every $m$-periodic $(\mcA,\mcB)$-Gorenstein object is weak. Moreover, from $\pd_{\mcB}(\mcA) = 0$ we can also give several characterizations of these objects by means of several acyclicity conditions. We prove in Proposition \ref{prop:characterizationGPvsWGP} that an object is $m$-periodic $(\mcA,\mcB)$-Gorenstein projective if, and only if, it is $(\mcA,\mcB)$-Gorenstein projective (in the sense of \cite{BMS}) and $\mcA$-periodic with period $m$ (a proposed generalization of the $\mcA$-periodic objects studied in \cite{BCE}). Some characterizations of $m$-periodic $(\mcA,\mcB)$-Gorenstein projective objects in terms of acyclicity conditions are proved in Theorem \ref{teo:equivmSGP}. This result is inspired in the equivalent descriptions of $m$-strongly Gorenstein projective modules proved by Bennis and Mahdou in \cite[Thm. 2.8]{BM}, and by Zhao and Huang in \cite[Thm. 3.9]{ZH}.

Another result from \cite{ZH} (namely, an $m$-strongly Gorenstein projective module is projective if, and only if, it is self-orthogonal) is proved in Section \ref{sec: orth relation} within our relative context for a particular class of $m$-periodic $(\mcA,\mcB)$-Gorenstein projective objects, which we call \emph{proper}, and defined as $mk$-th cycles in $\Hom(\mcA,-)$-acyclic and $\Hom(-,\mcB)$-acyclic complexes (see Definition \ref{def: piGPacyc} and Corollary \ref{cor: M in A<-> (MM)<m}). We also explore some relations between proper $m$-periodic $(\mcA,\mcB)$-Gorenstein projective objects and the $(n+1)$-cluster tilting subcategories defined by Iyama in \cite{Iyama}. The later family of subcategories allows to show examples of relative $m$-periodic Gorenstein projective objects which are not weak (see Corollary \ref{cor: SWGPclustertilting}).

\begin{figure}[H]
\includegraphics[width=1.15\textwidth]{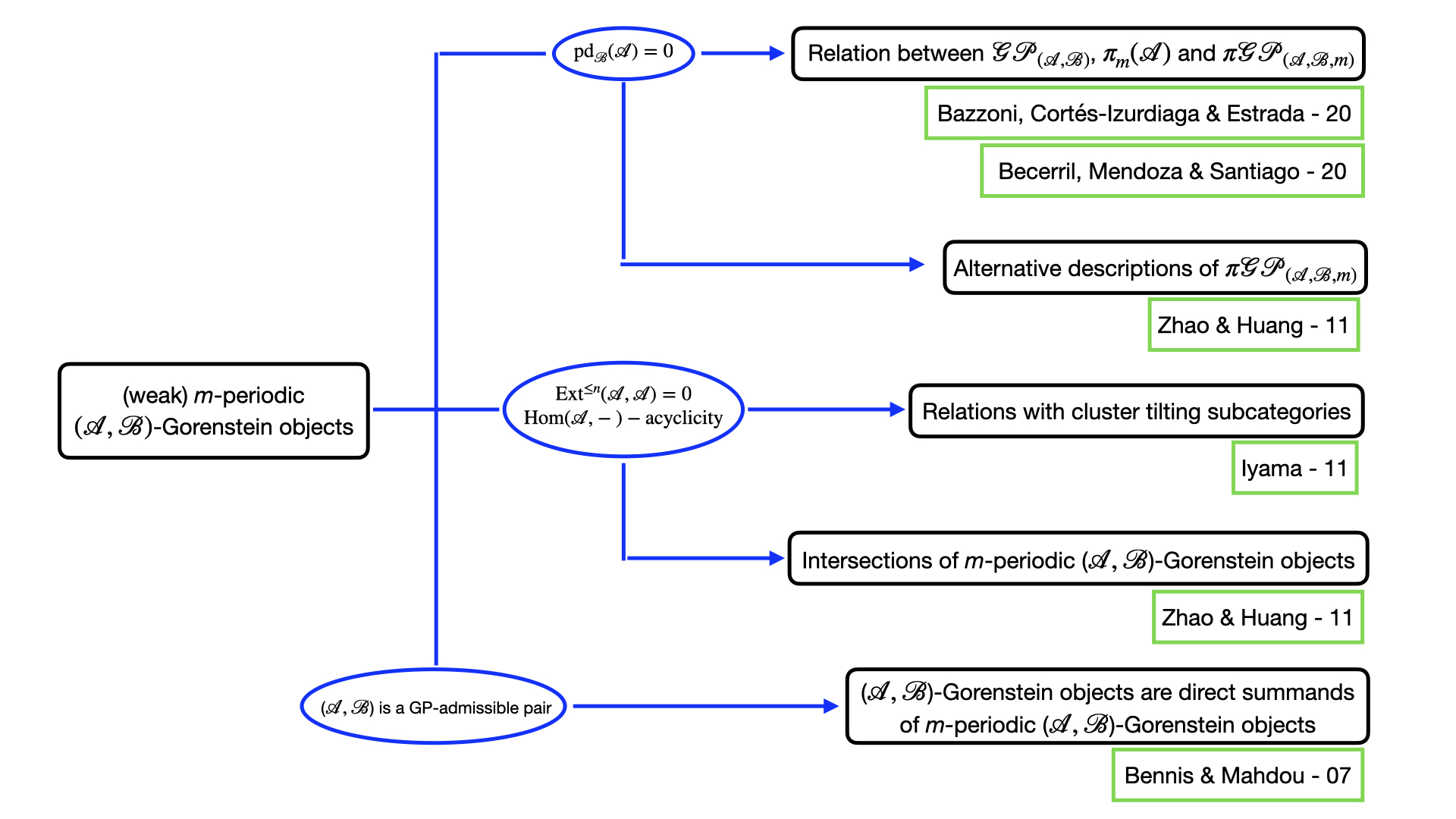}
\caption{A flowchart of the results obtained depending on the assumptions for $(\mcA,\mcB)$.}
\end{figure}

Proper $m$-periodic $(\mcA,\mcB)$-Gorenstein projective objects have more relations with the results from \cite{ZH} for $m$-strongly Gorenstein projective modules. In Section \ref{sec: intersections} we show that the intersection of proper $m$-periodic and proper $n$-periodic $(\mcA,\mcB)$-Gorenstein projective objects yields $(\mcA,\mcB)$-Gorenstein projective objects with period equal to the greatest common divisor of $m$ and $n$, as shown by Zhao and Huang in \cite{ZH} for the absolute case $\mcA = \mcB = \text{projective $R$-modules}$.

Finally, Section \ref{sec: admissible} is devoted to study the outcomes of considering $m$-periodic Gorenstein objects relative to a GP-admissible pair $(\mcA,\mcB)$. Perhaps the most important result in this section is the characterization of $(\mcA,\mcB)$-Gorenstein projective objects in terms of direct summands of $1$-periodic $(\mcA,\mcB)$-Gorenstein projective objects (see Proposition \ref{pro: GP sumandos directos de mGP}). This result will help us to give descriptions of the objects in the testing class $\mcB$ which are $(\mcA,\mcB)$-Gorenstein projective or (proper) $m$-periodic $(\mcA,\mcB)$-Gorenstein projective. Furthermore, it will be possible to deduce some properties concerning global and finitistic Gorenstein dimensions relative to GP-admissible and GI-admissible pairs (see Section \ref{sec:dimensions}).


\subsection*{Conventions}

Throughout, $\mcC$ will always denote an abelian category (not necessarily with enough projective and injective objects).
The main examples of such categories considered in this article will be:
\begin{itemize}
\item $\Mod(R)$ = left $R$-modules (with $R$ an associative ring with identity).

\item $\fgMod(\Lambda)$ = finitely generated left $\Lambda$-modules (with $\Lambda$ an Artin algebra). 

\item $\Ch(\mcC)$ = chain complexes of objects in $\mcC$. For the case where $\mcC = \Mod(R)$, the corresponding category of chain complexes of left $R$-modules will be denoted by $\Ch(R)$. Objects in $\Ch(\mcC)$ are denoted as $X_\bullet$, each $X_m$ denotes the $m$-th component of $X_\bullet$ in $\mcC$, and $Z_m(X_\bullet)$ denotes the $m$-th cycle of $X_\bullet$ in $\mcC$. 
\end{itemize}
Unless otherwise stated, all modules over $R$ and $\Lambda$ are assumed to be left $R$-modules and left $\Lambda$-modules, respectively. Similarly, complexes of modules are complexes of left $R$-modules. 

All classes of objects in $\mcC$ may be regarded as full subcategories of $\mcC$. Given a class $\mcA \subseteq \mcC$, we shall denote by ${\rm Free}(\mathcal{A})$ (resp., ${\rm free}(\mathcal{A})$) the class of all objects isomorphic to (finite) direct sums of objects in $\mcA$, and by ${\rm Add}(\mcA)$ (resp., $\add(\mcA)$) the class of all objects isomorphic to direct summands of objects in ${\rm Free}(\mcA)$ (resp., ${\rm free}(\mathcal{A})$). Monomorphisms and epimorphisms in $\mcC$ are denoted by using the arrows $\rightarrowtail$ and $\twoheadrightarrow$, respectively. We shall refer to commutative grids whose rows and columns are exact sequences as \emph{solid diagrams}. Finally, we point out that the definitions and results presented in this article have their corresponding dual statements, which in many cases will be omitted for simplicity.


\section{Preliminaries}\label{sec:preliminaries}

In the following concepts, $\mcA, \mcB \subseteq \mcC$ and $C \in \mcC$.


\subsection*{Resolution dimensions}

An \emph{$\mcA$-resolution of $C$} is an exact sequence
\[
\cdots \to A_k \to A_{k-1} \to \cdots \to A_1 \to A_0 \twoheadrightarrow C
\]
with $A_k \in \mcA$ for every $k \geq 0$. If $A_k = 0$ for every $k > m$, we shall say that the previous sequence is a \emph{finite $\mcA$-resolution of $C$ of length $m$}. The \emph{resolution dimension of $C$ with respect to $\mcA$} (or the \emph{$\mcA$-resolution dimension of $C$}, for short), denoted $\resdim_{\mcA}(C)$, is the smallest $m \geq 0$ such that $C$ admits a finite $\mcA$-resolution of length $m$. If such $m$ does not exist, we set $\resdim_{\mcA}(C) := \infty$. Dually, we have the concepts of (finite) \emph{coresolution} and \emph{coresolution dimension of $C$ with respect to $\mcA$}, denoted by $\coresdim_{\mcA}(C)$. With respect to these two homological dimensions, we shall frequently consider the following classes of objects in $\mcC$:
\begin{align*}
\mcA^\wedge_m & := \{ C \in \mcC \text{ : } \resdim_{\mcA}(C) \leq m \}, & & \text{and} & \mcA^\wedge & := \bigcup_{m \geq 0} \mcA^\wedge_m, \\
\mcA^\vee_m & := \{ C \in \mcC \text{ : } \coresdim_{\mcA}(C) \leq m \}, & & \text{and} & \mcA^\vee & := \bigcup_{m \geq 0} \mcA^\vee_m.
\end{align*}
Given a class $\mcB$ of objects in $\mcC$, the \emph{$\mcA$-resolution dimension of $\mcB$} is defined as
\[
{\rm resdim}_{\mcA}(\mcB) := {\rm sup}\{ {\rm resdim}_{\mcA}(B) {\rm \ : \ } B \in \mcB \}.
\]
The \emph{$\mcA$-coresolution dimension of $\mcB$} is defined dually, and denoted by ${\rm coresdim}_{\mcA}(\mcB)$.


\subsection*{Orthogonality with respect to extension functors} 

For $i \geq 1$, let 
\[
\Ext^i(-,-) \colon \mcC^{\rm op} \times \mcC \longrightarrow \Mod(\mathbb{Z})
\] 
denote the extension bifunctors in the sense of Yoneda. We shall identify $\Ext^0(-,-)$ with the hom bifunctor $\Hom(-,-)$. We fix some notations regarding the vanishing of $\Ext^i(-,-)$:
\begin{itemize}
\item $\Ext^i(\mcA,\mcB) = 0$ if $\Ext^i(A,B) = 0$ for every $A \in \mcA$ and $B \in \mcB$. In the case where $\mcA = \{ M \}$ or $\mcB = \{ N \}$, we shall write $\Ext^i(M,\mcB) = 0$ and $\Ext^i(\mcA,N) = 0$, respectively. 

\item For $i\geq 1$, $\Ext^{\leq i}(\mcA, \mcB) = 0$ if $\Ext^{j}(\mcA, \mcB) = 0$ for every $1 \leq j \leq i$. 

\item For $i \geq 1$, $\Ext^{\geq i}(\mcA, \mcB) = 0$ if $\Ext^{j}(\mcA, \mcB) = 0$ for every $j \geq i$. 
\end{itemize}
Recall that the \emph{right $i$-th orthogonal complement} and \emph{total orthogonal complement of $\mcA$} are defined by 
\[
\mcA^{\perp_i} := \{ N \in \mcC \mbox{ : } \Ext^i(\mcA,N) = 0 \}, \text{ \ } \mcA^\perp := \bigcap_{i \geq 1} \mcA^{\perp_i} = \{ N \in \mcC \mbox{ : } \Ext^{\geq 1}(\mcA,N) = 0 \},
\]
respectively. Dually, we have the \emph{$i$-th} and \emph{the total left orthogonal complements ${}^{\perp_i}\mcB$ and ${}^{\perp}\mcB$ of $\mcB$}, respectively. 


\subsection*{Relative homological dimensions}

The \emph{relative projective dimensions of $C$ and $\mcA$ with respect to $\mcB$} are defined by
\begin{align*}
\pd_{\mcB}(C) & := \inf \{ m \geq 0 {\rm \ : \ } \Ext^{\geq m+1}(C,\mcB) = 0 \}, \text{ \ } \pd_{\mcB}(\mcA) := \sup\{ \pd_{\mcB}(A) \text{ : } A \in \mcA \},
\end{align*}
respectively. If $\mcB = \mcC$, we just write $\pd(M)$ and $\pd(\mcA)$ for the (absolute) projective dimensions. Dually, we denote by $\id_{\mcA}(M)$ and $\id_{\mcA}(\mcB)$ the \emph{relative injective dimension of $M$} and \emph{$\mcB$}, respectively, \emph{with respect to $\mcA$}. If $\mcA = \mcC$, $\id(M)$ and $\id(\mcB)$ denote the (absolute) injective dimensions. It can be seen that 
\[
\pd_{\mcB}(\mcA) = \id_{\mcA}(\mcB).
\]


\subsection*{Resolving classes}

A class $\mcA$ is \emph{preresolving} if it is closed under extensions and under epikernels (that is, under taking kernels of epimorphisms between objects in $\mcA$). Moreover, if we let $\mcP$ denote the class of projective  objects in $\mcC$, a preresolving class $\mcA$ is \emph{resolving} if $\mcP \subseteq \mcA$. If the corresponding dual properties hold true, then $\mcA$ is called \emph{precoresolving} or \emph{coresolving}. For the latter, one considers the class $\mathcal{I}$ of injective objects in $\mathcal{C}$. 


\subsection*{Approximations}

A morphism $f \colon A \to C$ with $A \in \mcA$ is an \emph{$\mcA$-precover of $C$} if the induced group homomorphism 
\[
\Hom_{\mcC}(A',f) \colon \Hom_{\mcC}(A',A) \to \Hom_{\mcC}(A',C)
\] 
is epic for every $A' \in \mcA$. An $\mcA$-precover $f \colon A \to C$ is \emph{special} if it is epic and $\Ker(f) \in \mcA^{\perp_1}$. The dual concept is called (\emph{special}) \emph{$\mcA$-preenvelope}. Moreover, $\mcA$ is a \emph{precovering class} (resp., \emph{preenveloping class}) if every object in $\mcC$ has an $\mcA$-precover (resp., $\mcA$-preenvelope). Finally, $\mcA$ is \emph{functorially finite} if it is both precovering and preenveloping.


\subsection*{Relative cogenerators}
A class $\omega \subseteq \mcA$ is said to be a \emph{relative cogenerator in $\mcA$} if for every $A \in \mcA$ there is a short exact sequence $A \rightarrowtail W \twoheadrightarrow A'$ where $W \in \omega$ and $A' \in \mcA$. The dual concept is called \emph{relative generator in $\mcA$}.


\subsection*{Relative Gorenstein objects}

In what follows, given a chain complex $X_\bullet \in \Ch(\mcC)$, the notation $X_\bullet \in \Ch(\mcA)$ means that $X_m \in \mcA$ for every $m \in \mbZ$. We shall say that $X_\bullet$ is \emph{$\Hom(-,\mcB)$-acyclic} if the induced complex of abelian groups 
\[
\Hom(X_\bullet,B) = (\Hom(X_m,B))_{m \in \mbZ}
\] 
is exact for every $B \in \mcB$. Dually, we have the concept of $\Hom(\mcB,-)$-acyclic complexes. 

Most of our examples will be built from Gorenstein objects relative to certain pairs $(\mcA,\mcB)$ of classes of objects in $\mcC$. Following Becerril, Mendoza and Santiago's \cite[Defs. 3.2 \& 3.11]{BMS}, an object $C \in \mcC$ is:
\begin{itemize}
\item \emph{$(\mcA,\mcB)$-Gorenstein projective} if $C = Z_0(X_\bullet)$ for some exact and $\Hom(-,\mcB)$-acyclic complex $X_\bullet \in \Ch(\mcA)$. The \emph{$(\mcA,\mcB)$-Gorenstein injective objects} are defined dually. We borrow from \cite{BMS} the notations $\mcGP_{(\mcA,\mcB)}$ and $\mcGI_{(\mcA,\mcB)}$ for the classes of $(\mcA,\mcB)$-Gorenstein projective and $(\mcA,\mcB)$-Gorenstein injective objects of $\mcC$, respectively.

\item \emph{Weak $(\mcA,\mcB)$-Gorenstein projective} if $M \in {}^{\perp}\mcB$ and there is an exact sequence $M \rightarrowtail X_0 \to X_1 \to \cdots$ with $X_k \in \mcA$ for every $k \geq 0$ and with cycles in ${}^{\perp}\mcB$. The \emph{weak $(\mcA,\mcB)$-Gorenstein injective objects} are defined dually. The classes of such objects will be denoted by $\mathcal{WGP}_{(\mcA,\mcB)}$ and $\mathcal{WGI}_{(\mcA,\mcB)}$, respectively. 
\end{itemize}
In the case $\mcA = \mcB$, we write $\mathcal{GP}_{\mcA}$ and $\mathcal{GI}_{\mcA}$. For example, $\mcGP_{\mcP}$ and $\mcGI_{\mcI}$ are precisely the classes of Gorenstein projective and Gorenstein injective objects of $\mcC$, which we shall write as $\mcGP$ and $\mcGI$, for simplicity. As particular examples where $\mcA$ and $\mcB$ do not necessarily coincide, we have the following:
\begin{itemize}
\item Gillespie's \cite[Defs. 3.2 \& 3.7]{GillespieDC}: Ding projective and Ding injective modules, by setting $(\mcA,\mcB) = (\mcP(R),\mcF(R))$ and $(\mcA,\mcB) = (\text{FP-}\mcI(R),\mcI(R))$, respectively. Here, $\text{FP-}\mcI(R)$ stands for the class of FP-injective (or absolutely pure) $R$-modules, and $\mathcal{F}(R)$ denotes the class of flat $R$-modules. 

\item Bravo, Gillespie and Hovey's \cite[\S \ 5 \& 8]{BGH}: Gorenstein AC-projective and Gorenstein AC-injective modules, by setting $(\mcA,\mcB) = (\mcP(R),\mcL(R))$ and $(\mcA,\mcB) = (\text{FP}_\infty\text{-}\mcI(R),\mcI(R))$, respectively. Here, $\mcL(R)$ and $\text{FP}_\infty\text{-}\mcI(R)$ denote, respectively, the classes of level and $\text{FP}_\infty$-injective (or absolutely clean) $R$-modules (see also \cite[Def. 2.6]{BGH}). 
\end{itemize}


\section{Acyclicity of chain complexes under $\Hom(-,-)$}\label{sec:acyclicity}

The generalization of $n$-strongly Gorenstein projective and injective modules, which will be introduced in the following section, is a notion that depends on acyclicity conditions of complexes with respect to restrictions of the Hom functor. This section is devoted to study some of these conditions. Specifically, given classes of objects $\mathcal{A}$ and $\mathcal{B}$ in an abelian category $\mathcal{C}$, we consider chain complexes $A_\bullet \in \Ch(\mathcal{A})$ and look for necessary and sufficient conditions on the cycles $Z_k(A_\bullet)$ that characterize the acyclicity of $A_\bullet$ relative to $\Hom(\mathcal{A},-)$ and $\Hom(-,\mathcal{B})$. 

In what follows, according to Iyama's \cite[Def. 1.1]{Iyama}, we shall say that for $n \geq 1$ a class $\mcA \subseteq \mcC$ is:
\begin{itemize}
\item \emph{$(n+1)$-rigid} if $\Ext^{\leq n}(\mathcal{A,A}) = 0$;

\item \emph{$(n+1)$-cluster tilting} if it is functorially finite and satisfies 
\[
\mcA = \bigcap_{0 < i < n+1} {}^{\perp_i}\mcA = \bigcap_{0 < i < n+1} \mcA^{\perp_i}.
\]
In particular, $(n+1)$-cluster tilting subcategories are $(n+1)$-rigid.
\end{itemize}

Let us commence with the following well known relation between the $i$-fold extension groups $\Ext^i(Z_k(A_\bullet),B)$ and $\Ext^i(A,Z_k(A_\bullet))$ with $A \in \mathcal{A}$ and $B \in \mathcal{B}$.

\begin{lemma}[shifting]\label{lem:shifting}
Let $\mathcal{A}$ and $\mathcal{B}$ be classes of objects in $\mathcal{C}$, $A_\bullet \in \Ch(\mathcal{A})$ be an exact complex, and $n \geq 1$. Then, for every $k \in \mathbb{Z}$, $1 \leq i \leq n-1$, $A \in \mathcal{A}$ and $B \in \mathcal{B}$, one has the following natural isomorphisms:
\begin{align*}
\Ext^i(A,Z_{k}(A_\bullet)) & \cong \Ext^{i+1}(A,Z_{k+1}(A_\bullet)), \ \text{ provided that } \mathcal{A} \text{ is $(n+1)$-rigid}, \\
\Ext^i(Z_{k+1}(A_\bullet),B) & \cong \Ext^{i+1}(Z_k(A_\bullet),B), \ \text{ provided that } \Ext^{\leq n}(\mathcal{A,B}) = 0.
\end{align*}
\end{lemma}

We can represent inductively the previous natural isomorphisms in the following picture:
\begin{equation}\label{fig1} 
\parbox{7.75in}{ \scriptsize
\begin{tikzpicture}[description/.style={fill=white,inner sep=2pt}] 
\matrix (m) [ampersand replacement=\&, matrix of math nodes, row sep=0.5em, column sep=0em,text height=1.25ex, text depth=0.25ex] 
{ 
{} \& {} \& {} \& {} \& {} \& {} \& {} \& {} \& {}^1(Z_k(A_\bullet),B) \& \cong \& \cdots \\ 
{} \& {} \& {} \& {} \& {} \& {} \& {}^1(Z_{k+1}(A_\bullet),B) \& \cong \& {}^2(Z_k(A_\bullet),B) \& \cong \& \cdots  \& {} \& {} \\ 
{} \& {} \& {} \& {} \& {} \& {} \& \vdots \& {} \& \vdots \& {} \& {}  \& {} \& {} \\ 
{} \& {} \& {}^1(Z_{k+i-1}(A_\bullet),B) \& \cong \& \cdots \& \cong \& {}^{i-1}(Z_{k+1}(A_\bullet),B) \& \cong \& {}^i(Z_k(A_\bullet),B) \& \cong \& \cdots  \& {} \& {} \\ 
{}^1(Z_{k+i}(A_\bullet),B) \& \cong \& {}^2(Z_{k+i-1}(A_\bullet),B) \& \cong \& \cdots \& \cong \& {}^{i}(Z_{k+1}(A_\bullet),B) \& \cong \& {}^{i+1}(Z_k(A_\bullet),B) \& \cong \& \cdots  \& {} \& {} \\ 
\vdots \& {} \& \vdots \& {} \& {} \& {} \& \vdots \& {} \& \vdots \\
}; 
\path[->] 
; 
\end{tikzpicture} 
}
\end{equation}

The following is a relativization of Schanuel's Lemma (see for instance \cite[Ch. II, \S 5.A]{Lam}).

\begin{lemma}[Schanuel]\label{lem: Schanuel}
Let $\mcA$ be a class of objects in $\mcC$ closed under finite coproducts and satisfying $\Ext^1(\mcA,\mcA) = 0$. Let
$A_\bullet$ and $A'_{\bullet}$ be exact and $\Hom(\mcA,-)$-acyclic complexes in $\Ch(\mcA)$ such that  $Z_0(A_\bullet) \simeq Z_0(A'_\bullet)$. Then, for each $k > 0$ there exist objects $A, A'\in \mcA$ such that 
\[
Z_{k}(A_\bullet) \oplus A \simeq Z_{k}(A'_\bullet) \oplus A'.  
\]
\end{lemma}

\begin{proof}
First, since  $\Ext^{1}(\mcA, \mcA) = 0$ and $A_\bullet$ and $A'_\bullet$ are exact and $\Hom(\mathcal{A},-)$-acyclic, we have that $\Ext^{1}(\mcA, Z_{k}(A_{\bullet})) = 0$ and $\Ext^{1}(\mcA, Z_{k}(A'_{\bullet})) = 0$ for every $k\in \mathbb{Z}$. On the other hand, taking the pullback of the exact sequences $Z_1(A'_\bullet) \rightarrowtail A'_1 \twoheadrightarrow Z_0(A'_\bullet)$ and $Z_1(A_\bullet) \rightarrowtail A_1 \twoheadrightarrow Z_0(A_\bullet)$, where $Z_0(A_\bullet) \simeq Z_0(A'_\bullet)$, yields the following solid diagram:
\[
\begin{tikzpicture}[description/.style={fill=white,inner sep=2pt}] 
\matrix (m) [matrix of math nodes, row sep=2.5em, column sep=2.5em, text height=1.25ex, text depth=0.25ex] 
{ 
{} & Z_1(A_\bullet) & Z_1(A_\bullet) \\
Z_1(A'_\bullet) & E_1 & A_1 \\
Z_1(A'_\bullet) & A'_1 & Z_0(A'_\bullet) \\
}; 
\path[->] 
(m-2-2)-- node[pos=0.5] {\footnotesize$\mbox{\bf pb}$} (m-3-3) 
;
\path[>->]
(m-1-2) edge (m-2-2) (m-1-3) edge (m-2-3)
(m-2-1) edge (m-2-2) (m-3-1) edge (m-3-2)
;
\path[->>]
(m-2-2) edge (m-3-2) (m-2-3) edge (m-3-3)
(m-2-2) edge (m-2-3)
(m-3-2) edge (m-3-3)
;
\path[-,font=\scriptsize]
(m-2-1) edge [double, thick, double distance=2pt] (m-3-1)
(m-1-2) edge [double, thick, double distance=2pt] (m-1-3)
;
\end{tikzpicture} .
\]
Since $\Ext^{1}(\mcA, Z_{k}(A_{\bullet}))=0=\Ext^{1}(\mcA, Z_{k}(A'_{\bullet}))$ for 
every $k\in \mathbb{Z}$, the central row and column split, and so $Z_1(A'_\bullet) \oplus A_1 \simeq Z_1(A_\bullet) \oplus A'_1$. This in turn implies the existence of the following solid diagram: 
\[
\begin{tikzpicture}[description/.style={fill=white,inner sep=2pt}] 
\matrix (m) [matrix of math nodes, row sep=2.5em, column sep=2.5em, text height=1.25ex, text depth=0.25ex] 
{ 
{} & Z_2(A_\bullet) & Z_2(A_\bullet) \\
Z_2(A'_\bullet) & E_2 & A_2 \oplus A'_1 \\
Z_2(A'_\bullet) & A'_2 \oplus A_1 & Z_1(A_\bullet) \oplus A'_1 \\
}; 
\path[->] 
(m-2-2)-- node[pos=0.5] {\footnotesize$\mbox{\bf pb}$} (m-3-3) 
;
\path[>->]
(m-1-2) edge (m-2-2) (m-1-3) edge (m-2-3)
(m-2-1) edge (m-2-2) (m-3-1) edge (m-3-2)
;
\path[->>]
(m-2-2) edge (m-3-2) (m-2-3) edge (m-3-3)
(m-2-2) edge (m-2-3)
(m-3-2) edge (m-3-3)
;
\path[-,font=\scriptsize]
(m-2-1) edge [double, thick, double distance=2pt] (m-3-1)
(m-1-2) edge [double, thick, double distance=2pt] (m-1-3)
;
\end{tikzpicture} .
\]
Using again $\Ext^{1}(\mcA, Z_{k}(A_{\bullet}))=0=\Ext^{1}(\mcA, Z_{k}(A'_{\bullet}))$ for 
every $k\in \mathbb{Z}$, we get that $Z_2(A_\bullet) \oplus A'_2 \oplus A_1 \simeq Z_2(A'_\bullet) \oplus A_2 \oplus A'_1$. Repeating this procedure inductively,
the result follows.
\end{proof}

For the rest of this section, we shall focus on the following particular type of bounded chain complexes. In what follows, $m \geq 1$ is a positive integer.

\begin{definition}\label{def:loop}
Given $\mathcal{A} \subseteq \mathcal{C}$ and $M \in \mathcal{C}$, an \textbf{$\bm{\mathcal{A}}$-loop at $\bm{M}$ of length $\bm{m}$} is an exact chain complex $A_\bullet \in \Ch(\mathcal{A})$ such that $Z_{k m}(A_\bullet) \simeq M$ for every $k \in \mathbb{Z}$. 
\end{definition}

\begin{remark}\label{rem:loop}
Notice that from any $\mathcal{A}$-loop at $M$ of length $m$ one obtains a bounded exact chain complex of the form
\begin{align}
0 & \to M \to A_m \to \cdots \to A_1 \to M \to 0 \label{eqn:loop}
\end{align}
with $A_k \in \mathcal{A}$ for every $1 \leq k \leq m$. Conversely, glueing at $M$ infinitely many copies of \eqref{eqn:loop} yields an $\mathcal{A}$-loop at $M$ of length $m$. 
\end{remark}

The following is a consequence of Lemma \ref{lem:shifting}.

\begin{corollary}\label{cor:Li^perp=M^perp}
Let $\mcA, \mcB \subseteq \mcC$ such that $\pd_{\mcB}(\mcA) = 0$, and $A_\bullet \in \Ch(\mathcal{C})$ be an $\mathcal{A}$-loop at $M \in \mathcal{C}$ of length $m$. Then, the following are equivalent for $\mathcal{Y} = \mathcal{B}, \mathcal{B}^\wedge$: 
\begin{enumerate}[(a)]
\item $M\in {}^{\perp}\mcY$;

\item $Z_{k}(A_\bullet) \in {}^{\perp}\mcY$ for every $1\leq k\leq m$;

\item $Z_{k}(A_\bullet) \in {}^{\perp}\mcY$ for some $1\leq k\leq m$;

\item $Z_{k}(A_\bullet) \in {}^{\perp_{1}}\mcY$ for every $1\leq k\leq m$;

\item There exists $i \geq 1$ such that $Z_{k}(A_\bullet) \in {}^{\perp_{i}}\mcY$ for every $1\leq k \leq m$;

\item There exists $i\geq 0$ such that $M \in \bigcap^{m}_{k = 1} {}^{\perp_{i+k}}\mcY$.
\end{enumerate}
\end{corollary}

\begin{proof}
The validity of the case $\mathcal{Y} = \mathcal{B}$ can be traced back using the diagram \eqref{fig1}. This, along with \cite[Prop. 2.6]{HMP}, implies the case $\mathcal{Y} = \mathcal{B}^\wedge$. 
\end{proof}

Acyclicity of $\mathcal{A}$-loops under $\Hom(\mcA,-)$ can be characterized in terms of its cycles for the case where $\mathcal{A}$ is rigid, as the following result shows.

\begin{proposition}\label{prop:ex}
Let $\mcA \subseteq \mcC$ be an $(n+1)$-rigid class for $n\geq 1$, and $A_\bullet \in \Ch(\mcC)$ be an $\mcA$-loop at $M \in \mcC$ of length $m$. Consider the following assertions:
\begin{enumerate}[(a)]
\item The bounded complex \eqref{eqn:loop} is $\Hom(\mcA, -)$-acyclic.

\item $Z_{k}(A_\bullet) \in \bigcap_{i=1}^{n}\mcA^{\perp_{i}}$ for every $1\leq k\leq m$.

\item $Z_k(A_\bullet) \in \mcA^{\perp_1}$ for every $1 \leq k \leq m$.

\item $M \in \bigcap_{i=1}^{n} \mcA^{\perp_{i}}$.
\end{enumerate}
Then, the implications (a) $\Leftrightarrow$ (b) $\Leftrightarrow$ (c) $\Rightarrow$ (d) hold true. If in addition $m \leq n$, then all the assertions are equivalent.
\end{proposition}

\begin{proof} \
\begin{itemize}
\item (a) $\Rightarrow$ (b): Suppose first that \eqref{eqn:loop} is $\Hom(\mcA, -)$-acyclic, and consider for each $1\leq k\leq m$ the short exact sequence $Z_{k}(A_\bullet) \rightarrowtail A_{k} \twoheadrightarrow Z_{k-1}(A_\bullet)$. Lemma \ref{lem:shifting} implies that $\Ext^i(\mcA, Z_{k-1}(A_\bullet)) \cong \Ext^{i+1}(\mcA, Z_k(A_\bullet))$ for every $1 \leq i \leq n-1$. One the other hand, the $\Hom(\mcA,-)$-acyclicity of the previous sequence yields $\Ext^1(\mcA, Z_{k}(A_\bullet)) = 0$. So the result follows. 

\item (b) $\Rightarrow$ (c) $\Rightarrow$ (a): Immediate. 

\item (c) $\Rightarrow$ (d): We need to consider the cases $n \leq m$ and $n > m$. The former follows by the dual of \eqref{fig1}. On the other hand, since $A_\bullet$ is an $\mcA$-loop at $M$ of length $m$, it is also an $\mcA$-loop at $M$ of length $mq$ for every $q > 0$. Thus, the case $n > m$ follows as the previous one by taking $q > 0$ such that $n \leq mq$. 

\item (d) $\Rightarrow$ (c): It is a consequence of the dual of \eqref{fig1} for the case $m \leq n$.
\end{itemize}
\end{proof}







\section{Relative (weakly) periodic Gorenstein objects}\label{sec: m-period}

In \cite{BM}, Bennis and Mahdou introduce a generalization for strongly Gorenstein projective modules (a notion also due to them, see \cite{BM2}) called $n$-strongly Gorenstein projective modules. In this section, we present the concept of $m$-periodic Gorenstein projective and injective objects relative to a pair $(\mathcal{A,B})$ of classes of objects in an abelian category $\mcC$, with the aim to extend the  concepts in \cite{BM2,BM} and study their properties resulting from several assumptions on $(\mathcal{A,B})$. The language of loop complexes will be key towards this goal. 

We propose the following relativization of $n$-strongly Gorenstein projective modules \cite[Def. 3.1]{BM}.

\begin{definition}\label{def:sGP}
We say that an object $M \in \mcC$ is \textbf{periodic $\bm{(\mcA,\mcB)}$-Gorenstein projective of period $\bm{m}$} (or \textbf{$\bm{m}$-periodic $\bm{(\mcA,\mcB)}$-Gorenstein projective}, for short) if there exists a $\Hom(-,\mcB)$-acyclic $\mcA$-loop at $M$ of length $m$. We denote by $\pi\mathcal{GP}_{(\mcA,\mcB,m)}$ the class of $m$-periodic $(\mcA,\mcB)$-Gorenstein projective objects in $\mcC$. Equivalently, we can say that $M \in \pi\mathcal{GP}_{(\mcA,\mcB,m)}$ if there exists a bounded exact and $\Hom(-,\mcB)$-acyclic complex as \eqref{eqn:loop}. The \textbf{$\bm{m}$-periodic $\bm{(\mcA,\mcB)}$-Gorenstein injective objects} are defined dually, and form a class of objects denoted by $\pi\mathcal{GI}_{(\mcA,\mcB,m)}$.
\end{definition}

We shall also work with the following particular family of relative $m$-periodic Gorenstein objects.

\begin{definition}\label{def:wsGP}
We say that $M \in \mcC$ is \textbf{weakly $\bm{m}$-periodic $\bm{(\mcA,\mcB)}$-Gorenstein projective} if there exists an $\mcA$-loop at $M$ of length $m$ with cycles in ${}^{\perp}\mcB$. We denote the class of weakly $m$-periodic $(\mcA,\mcB)$-Gorenstein projective objects by $\pi\mathcal{WGP}_{(\mcA,\mcB,m)}$. The \textbf{weakly $\bm{m}$-periodic $\bm{(\mcA,\mcB)}$-Gorenstein injective objects} are defined dually, and the class of such objects will be denoted by $\pi\mathcal{WGI}_{(\mcA,\mcB,m)}$.\footnote{The term ``weakly'' has to do with the concept of weak $(\mcA,\mcB)$-Gorenstein projective objects studied in \cite[Def. 3.11]{BMS}. The relation
between these notions will be studied in the following section.}
\end{definition}

\begin{remark}\label{rem:sGP}
{} \
\begin{enumerate}
\item For any ring $R$ and setting $\mcA = \mcB = \mcP(R)$, we have that $\pi\mathcal{GP}_{(\mcA,\mcB,m)}$ coincides with the class of $m$-strongly Gorenstein projective $R$-modules.

\item For any $m \geq 1$, one has the containments 
\[
\pi\mathcal{WGP}_{(\mcA,\mcB,1)} \subseteq \pi\mathcal{WGP}_{(\mcA,\mcB,m)} \text{ \ and \ }  \pi\mathcal{GP}_{(\mcA,\mcB,1)} \subseteq \pi\mathcal{GP}_{(\mcA,\mcB,m)},
\]
which may be strict (see Example \ref{ex: piGPm->piGP1}). 

\item The containment 
\[
\pi\mathcal{WGP}_{(\mcA,\mcB,m)} \subseteq \pi\mathcal{GP}_{(\mcA,\mcB,m)}
\] 
holds for any pair of classes $\mcA,\mcB \subseteq \mcC$. Indeed, it suffices to note that each sequence $Z_j(A_\bullet) \rightarrowtail A_j \twoheadrightarrow Z_{j-1}(A_\bullet)$ with $Z_j(A_\bullet), Z_{j-1}(A_\bullet) \in {}^{\perp}\mcB$ is $\Hom(-,\mcB)$-acyclic. This containment may be strict, as we shall see in Corollary~\ref{cor: SWGPclustertilting}. 
\begin{figure}[H]
\centering
\includegraphics[width=0.75\textwidth]{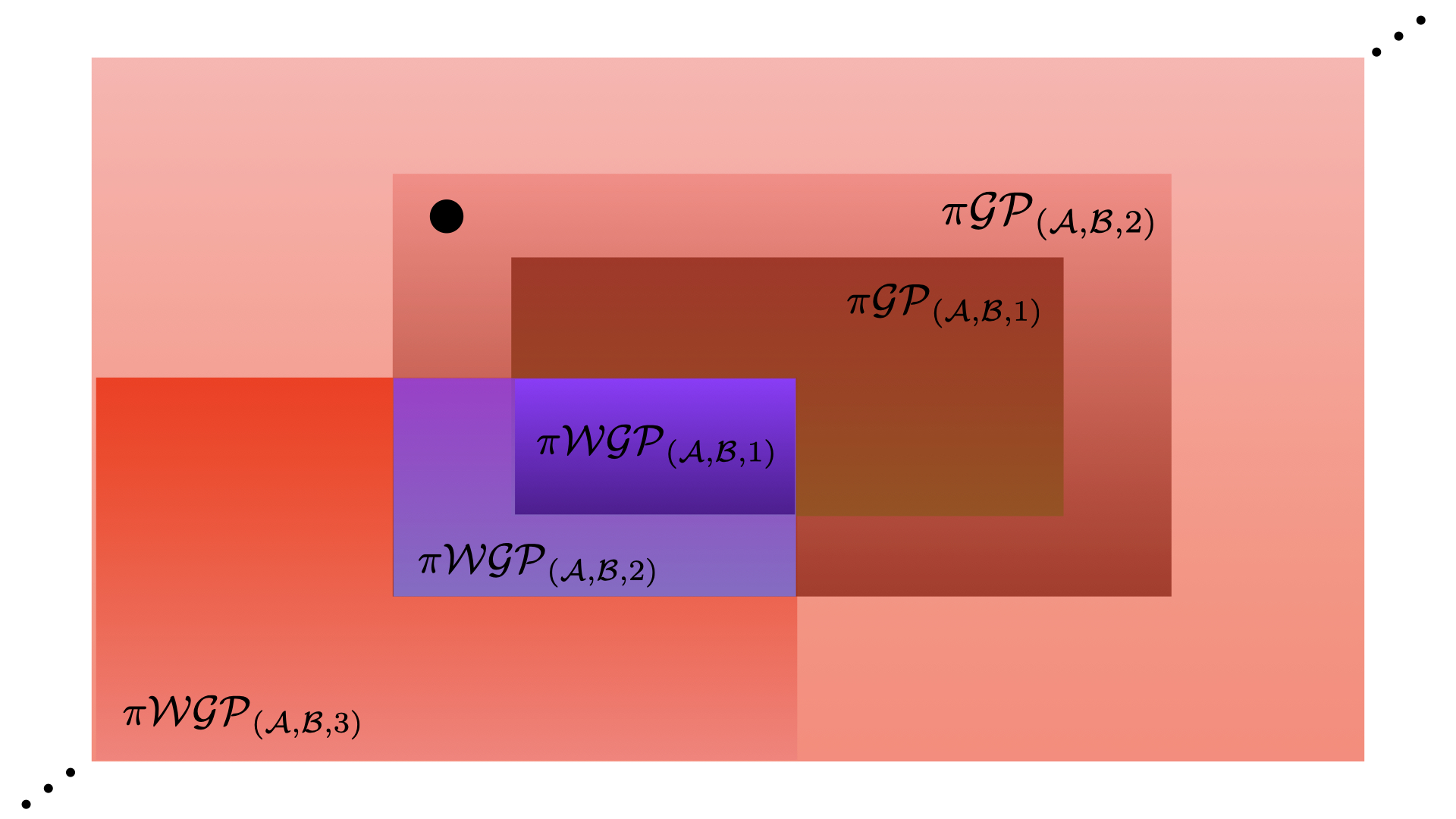}
\caption{A depiction of the containment relations in (2) and (3) for periodic and weak periodic Gorenstein objects. Note that the equality $\pi\mathcal{WGP}_{(\mathcal{A,B},2)} \cap \pi\mathcal{GP}_{(\mathcal{A,B},1)} = \pi\mathcal{WGP}_{(\mathcal{A,B},1)}$ is represented in the figure. On the other hand, the dot represents a 2-periodic $(\mathcal{A,B})$-Gorenstein projective object which is not 1-periodic $(\mathcal{A,B})$-Gorenstein projective (see for instance Example \ref{ex: piGPm->piGP1} and Corollary \ref{cor: SWGPclustertilting}).}
\end{figure}

\item The containment 
\begin{align}
\mcA & \subseteq \pi\mathcal{GP}_{(\mcA,\mcB,m)} \label{eqn:AGP}
\end{align}
holds for any even positive integer $m$. Indeed, in the case $m = 2$ for instance, we have for every $A \in \mcA$ the exact and $\Hom(-,\mcB)$-acyclic sequence
\[
A \stackrel{\mathrm{Id}_{A}}\rightarrowtail A \stackrel{0}\to A \stackrel{\mathrm{Id}_{A}}\twoheadrightarrow A,
\]
which shows that $A \in \pi\mathcal{GP}_{(\mathcal{A, B},2)}$. If in addition $\Ext^{\geq 1}(\mcA,\mcB) = 0$, then we also have that 
\begin{align}
\mcA & \subseteq \pi\mathcal{WGP}_{(\mcA,\mcB,m)} \subseteq \pi\mathcal{GP}_{(\mcA,\mcB,m)}. \label{eqn:AWGP}
\end{align}

If $\mcA$ is closed under finite coproducts, we can note that $\mcA \subseteq \pi\mathcal{GP}_{(\mcA,\mcB,1)}$ and  $\mcA \subseteq \pi\mathcal{WGP}_{(\mcA,\mcB,1)}$ (provided $\Ext^{\geq 1}(\mcA,\mcB) = 0$). Hence, by (2) we have that \eqref{eqn:AGP} and \eqref{eqn:AWGP} are valid for every $m \geq 1$ under this assumption. Moreover, \eqref{eqn:AGP} may be strict as it was shown in \cite[Ex. 2.5]{BM2} by taking $m = 1$ and $\mcA = \mcB = \mcP(R)$. For the relative case with $m>1$, see Example \ref{ex: piGPm->piGP1}.

\item One has the implications
\begin{align*} 
\mcB \subseteq \mcB' & \Rightarrow \pi\mathcal{GP}_{(\mcA,\mcB',m)} \subseteq \pi\mathcal{GP}_{(\mcA,\mcB,m)}, \\
\mcA \subseteq \mcA' & \Rightarrow \pi\mathcal{GI}_{(\mcA',\mcB,m)} \subseteq \pi\mathcal{GI}_{(\mcA,\mcB,m)}.
\end{align*}

\item The class $\pi\mathcal{GP}_{(\mcA,\mcB,m)}$ is not uniquely determined by the pair $(\mcA,\mcB)$ (see   Example \ref{ex:sGP}-(3)). 

\item If $M \in \pi\mathcal{GP}_{(\mcA,\mcB,m)}$ and $A_\bullet$ is a $\Hom(-,\mcB)$-acyclic $\mcA$-loop at $M$ of length $m$, then $Z_k(A_\bullet) \in \pi\mathcal{GP}_{(\mcA,\mcB,m)}$ for every $k \in \mathbb{Z}$.
\end{enumerate}
\end{remark}

\begin{example}\label{ex:sGP}
{} \
\begin{enumerate}
\item The facts mentioned in this example come from Zhang and Xiong's \cite[Ex. 5.3-(2)]{ZX}. Let $k$ be a field and $\Lambda$ be the path algebra over $k$ given by the quiver 
\[
\xymatrix{
1 \ar@/^/[r]^{\alpha} & 2\ar@/^/[l]^{\beta} & 3\ar[l]^{\gamma}
}
\]
with relations $\alpha \beta = 0 = \beta \alpha$. The indecomposable projective $\Lambda$-modules are 
\begin{align*}
P(1) & = \begin{tiny}\begin{array}{c} 1 \\ 2 \end{array} \end{tiny}, & P(2) & = \begin{tiny} \begin{array}{c} 2 \\ 1 \end{array} \end{tiny}, & & \text{and} & P(3) & = \begin{tiny} \begin{array}{c} 3 \\ 2 \\ 1 \end{array} \end{tiny},
\end{align*}
while the indecomposable injective $\Lambda$-modules are given by
\begin{align*}
I(1) & = \begin{tiny}\begin{array}{c} 3 \\ 2 \\ 1 \end{array}\end{tiny}, & I(2) & = \begin{tiny}\begin{array}{ccc} 1 & {} & 3 \\ {} & 2 \end{array}\end{tiny}, & & \text{and} & I(3) & = \begin{tiny}\begin{array}{c} 3 \end{array}\end{tiny}.
\end{align*}
On the other hand, the Auslander-Reiten quiver of $\Lambda$ is
\[
\xymatrix@R=0.3cm{
 & & {\begin{tiny}
\begin{array}{c}
3 \\
2\\
1
\end{array}
\end{tiny}}\ar[rd] & & & \\
 & {\begin{tiny}
\begin{array}{c}
2\\
1
\end{array}
\end{tiny}}\ar[ru] \ar[rd] \ar@{--}[rr] & & {\begin{tiny}
\begin{array}{c}
3\\
2
\end{array}
\end{tiny}}\ar[rd] \ar[rd] \ar@{--}[rr] & & \begin{tiny}
\begin{array}{c}1\end{array}
\end{tiny} \\
\begin{tiny}
\begin{array}{c} 1 \end{array}
\end{tiny} \ar[ru] \ar@{--}[rr] & & \begin{tiny}
\begin{array}{c} 2 \end{array}
\end{tiny} \ar[ru] \ar[rd] \ar@{--}[rr] & & {\begin{tiny}
\begin{array}{ccc}
1& &3\\
&2&
\end{array}
\end{tiny}}\ar[ru] \ar[rd] & \\
 & & & {\begin{tiny}
\begin{array}{c}
1\\
2
\end{array}
\end{tiny}}\ar[ru] \ar@{--}[rr] & & \begin{tiny}
\begin{array}{c} 3 \end{array}
\end{tiny}
}
\]
where the vertices $i$ represent the same simple module $S(i)$. Consider the class 
\[
\mcX = \add(\begin{tiny}
\begin{array}{c}
1
\end{array}
\end{tiny} \oplus \begin{tiny}
\begin{array}{c}
2\\
1
\end{array}
\end{tiny} \oplus \begin{tiny}
\begin{array}{c}
2
\end{array}
\end{tiny} \oplus \begin{tiny}
\begin{array}{c}
1\\
2
\end{array}
\end{tiny})
\] 
of $\Lambda$-modules isomorphic to direct summands of finite direct sums of $\begin{tiny}
\begin{array}{c}
1
\end{array}
\end{tiny} \oplus \begin{tiny}
\begin{array}{c}
2\\
1
\end{array}
\end{tiny} \oplus \begin{tiny}
\begin{array}{c}
2
\end{array}
\end{tiny} \oplus \begin{tiny}
\begin{array}{c}
1\\
2
\end{array}
\end{tiny}$. 
Then, $\mcX$ is a class of objects in $\modu(\Lambda)$ closed under extensions. Moreover, $\mcX$ is a Frobenius subcategory of $\modu(\Lambda)$ where the indecomposable projective-injective objects are precisely $\begin{tiny}
\begin{array}{c}
1\\
2
\end{array}
\end{tiny}$ and $\begin{tiny}
\begin{array}{c}
2 \\
1
\end{array}
\end{tiny}$. Thus, the class of projective objects in $\mcX$ is given by 
\[
\mcP(\mcX) = \add(P(1) \oplus P(2)).
\] 
Note also that 
\[
\mcP(\modu(\Lambda)) = \add(P(1) \oplus P(2) \oplus P(3)).
\]

If we set $\mcA = \mcB = \mcP(\mcX)$, note that $P(3)$ is a $1$-strongly Gorenstein projective $\Lambda$-module by Remark \ref{rem:sGP}-(4). However $P(3) \not\in \pi\mathcal{GP}_{(\mcA, \mcB, m)}$ for any $m\geq 1$. Indeed,
if we suppose that $P(3) \in \pi\mathcal{GP}_{(\mcA, \mcB, m)}$ for some $m \geq 1$ there is a short exact sequence 
\[
P(3) \rightarrowtail A_{m} \twoheadrightarrow K
\] 
where $A_m \in \mcA=\mcP(\mcX)$, which is split since $P(3) = I(1)$. It follows that $P(3) \in \mcX$, getting a contradiction. 

\item Recall that an $R$-module is pure projective if it is a direct summand of a direct sum of finitely presented modules (see Warfield's \cite[Coroll. 3]{Warfieldpurity}). If $\mathcal{PP}(R)$ denotes the class of pure projective $R$-modules, then $\mathcal{PP}(R) = {\rm Add}(\mathcal{FP}_1(R))$, where $\mathcal{FP}_1(R)$ denotes the class of finitely presented $R$-modules. We can note that any $m$-periodic $(\mathcal{FP}_1(R),\mathcal{PP}(R))$-Gorenstein injective $R$-module is pure projective, that is,
\[
\pi\mathcal{GI}_{(\mathcal{FP}_1(R),\mathcal{PP}(R),m)} \subseteq \mathcal{PP}(R).
\]
To show this, suppose we are given an $R$-module $M \in \pi\mathcal{GI}_{(\mathcal{FP}_1(R),\mathcal{PP}(R),m)}$, so there is an exact and pure exact sequence 
\[
M \rightarrowtail P_m \to \cdots \to P_1 \twoheadrightarrow M
\]
where each $P_k$ is pure projective. It follows by Bazzoni, Cort\'es-Izurdiaga and Estrada's \cite[Thm. 5.1-(3)]{BCE} that $M$ is pure projective. The other containment holds as well (see the dual of Remark \ref{rem:sGP}-(4)).

\item Let $\mcC(R) = \mcF(R)^{\perp_1}$ denote the class of cotorsion modules. We show that 
\[
\pi\mathcal{GI}_{(\mcF(R),\mcC(R),m)} = \pi\mathcal{GI}_{(\mcP(R),\mcC(R),m)} = \mcC(R).
\]
The containment $\pi\mathcal{GI}_{(\mcF(R),\mcC(R),m)} \subseteq \pi\mathcal{GI}_{(\mcP(R),\mcC(R),m)}$ is clear by Remark \ref{rem:sGP}-(5). On the other hand, for every $M \in \pi\mathcal{GI}_{(\mcP(R),\mcC(R),m)}$ we have $M \simeq Z_{km}(C_\bullet)$ for every $k \in \mathbb{Z}$, where $C_\bullet$ is a cotorsion loop  of length $m$ (in particular, an exact complex of cotorsion $R$-modules). By \cite[Thm. 5.1-(2)]{BCE}, we also know that $C_\bullet$ has cotorsion cycles. In particular, $M \in \mcC(R)$. The containment $\mcC(R) \subseteq \pi\mathcal{GI}_{(\mcF(R),\mcC(R),m)}$ follows by the dual of Remark \ref{rem:sGP}-(4).

\item In (2), let us replace $\mathcal{PP}(R)$ by the class $\mathcal{I}(R)$ of injective $R$-modules. For every $m \geq 1$, one has that every $m$-periodic $(\mathcal{FP}_1(R),\mathcal{I}(R))$-Gorenstein injective $R$-module is FP-injective. Indeed, for every $M \in \pi\mathcal{GI}_{(\mathcal{FP}_1(R),\mathcal{I}(R),m)}$ there is a pure exact complex $\cdots \to I_2 \to I_1 \twoheadrightarrow M$ where $I_k$ is injective for every $k \geq 1$, implying thus that $\Ext^1_{R}(F,M) = 0$ for every finitely presented $R$-module $F$. Hence, $M$ is FP-injective. In other words, we have the containment 
\begin{align*}
\pi\mathcal{GI}_{(\mathcal{FP}_1(R),\mathcal{I}(R),m)} & \subseteq {\rm FP}\mbox{-}\mcI(R).
\end{align*}
The converse containment does not hold in general, and its validity is part of the following characterization of noetherian rings.
\end{enumerate}
\end{example}

\begin{proposition}\label{prop:characterization_noetherian}
The following assertions are equivalent:
\begin{enumerate}[(a)]
\item $R$ is a (left) noetherian ring. 

\item ${\rm FP}\mbox{-}\mcI(R) = \mcI(R)$.

\item ${\rm FP}\mbox{-}\mcI(R) = \pi\mathcal{GI}_{(\mathcal{FP}_1(R),\mcI(R),m)}$ for every $m \geq 1$

\item ${\rm FP}\mbox{-}\mcI(R) = \pi\mathcal{GI}_{(\mathcal{FP}_1(R),\mcI(R),m)}$ for some $m \geq 1$. 
\end{enumerate}
\end{proposition}

\begin{proof}
The equivalence between (a) and (b) is proven, for instance, in Megibben's \cite[Thm. 3]{Megibben}. The implication (b) $\Rightarrow$ (c) follows by Example~\ref{ex:sGP}-(4) since
\[
\mcI(R)\subseteq \pi\mathcal{GI}_{(\mathcal{FP}_1(R),\mcI(R),m)} \subseteq {\rm FP}\mbox{-}\mcI(R)=\mcI(R),
\]
while (c) $\Rightarrow$ (d) is clear. Let us show (d) $\Rightarrow$ (b). So suppose every FP-injective $R$-module $E$ is $m$-periodic $(\mathcal{FP}_1(R),\mcI(R))$-Gorenstein injective. Then, $E \simeq Z_{0}(I_\bullet)$ for some $\Hom(\mathcal{FP}_1(R),-)$-acyclic injective loop $I_\bullet$ of length $m$. In particular, $I_\bullet$ is an exact complex of injective $R$-modules with FP-injective cycles. By \cite[Thm. 5.1-(1)]{BCE}, $I_\bullet$ is injective, and hence $E$ is an injective $R$-module (being a cycle of an injective complex).
\end{proof}


\section{Relation with the condition $\pd_{\mcB}(\mcA) = 0$ and periodic objects}\label{sec: hereditary}

In the following lines we prove some properties of $m$-periodic Gorenstein objects related to pairs $(\mathcal{A,B})$ satisfying $\pd_{\mcB}(\mcA) = 0$. First, we show that this condition characterizes when the classes of $m$-periodic and weakly $m$-periodic Gorenstein projective objects, relative to $(\mathcal{A,B})$, coincide.

\begin{proposition}\label{pro:WSGP=SGP}
The following are equivalent:
\begin{enumerate}[(a)]
\item ${\rm pd}_{\mcB}(\mcA) = 0$;

\item $\pi\mathcal{WGP}_{(\mathcal{A, B},m)}=\pi\mathcal{GP}_{(\mathcal{A, B},m)}$ for any $m\geq 1$;

\item $\pi\mathcal{WGP}_{(\mathcal{A, B},2)}=\pi\mathcal{GP}_{(\mathcal{A, B},2)}$.
\end{enumerate}
\end{proposition}

\begin{proof} 
The implication (a) $\Rightarrow$ (b) follows by Corollary \ref{cor:Li^perp=M^perp} and  Remark \ref{rem:sGP}-(3), while (b) $\Rightarrow$ (c) is clear. Finally, for (c) $\Rightarrow$ (a) is a consequence of Remark \ref{rem:sGP}-(4).  
\end{proof}

The previous result and Corollary~\ref{cor:Li^perp=M^perp} allow us to give an alternative description for $m$-periodic $(\mcA, \mcB)$-Gorenstein projective objects by combining $\Hom(-, \mcB)$-acyclic exact complexes and the vanishing of extension groups related to $\mcB^{\wedge}$.

\begin{corollary}\label{characterization2}
Let $\mcA, \mcB \subseteq \mcC$ such that $\pd_{\mcB}(\mcA) = 0$. Then, $M \in \pi\mathcal{GP}_{(\mathcal{A, B},m)}$ if, and only if, there exists an $\mcA$-loop $A_\bullet$ at $M$ of length $m$, and an integer $i > 0$ such that $\Ext^i(Z_k(A_\bullet),\mcB^\wedge) = 0$ for every $1 \leq k \leq m$. 
\end{corollary}

Following \cite{BCE}, let us propose the following generalization of periodic objects with respect to a class.

\begin{definition}
We say that an object $M \in \mcC$ is \textbf{$\bm{\mcA}$-periodic of period $\bm{m}$} if there exists an $\mcA$-loop at $M$ of length $m$. The class of these objects will be denoted by $\pi_m(\mcA)$. \\
\end{definition}

\begin{remark} ~\
\begin{enumerate}
\item In the case $\mcC := \Mod(R)$ and $m = 1$, $\mcA$-periodic modules of period $1$ are precisely the $\mcA$-periodic modules from \cite{BCE}. 

\item Every $m$-periodic $(\mcA,\mcB)$-Gorenstein projective object is clearly $\mcA$-periodic of period $m$. The converse is not true in general.
\end{enumerate}
\end{remark}

The following relation between $\mcA$-periodic objects of period $m$ and $m$-periodic $(\mcA,\mcB)$-Gorenstein projective objects is a consequence of \cite[Prop. 3.14]{BMS}, Corollary~\ref{cor:Li^perp=M^perp} and Remark~\ref{rem:sGP}-(3).

\begin{proposition}\label{prop:characterizationGPvsWGP}
Let $\mcA,\mcB \subseteq \mcC$ such that $\pd_{\mcB}(\mcA) = 0$. Then,
\[
\mathcal{GP}_{(\mcA,\mcB)} \cap \pi_m(\mcA) = \mathcal{WGP}_{(\mcA, \mcB)}\cap \pi_{m}(\mcA) = \pi\mathcal{WGP}_{(\mcA, \mcB, m)} = \pi\mathcal{GP}_{(\mcA, \mcB, m)}.
\]
\end{proposition}

In \cite[Thm. 2.8]{BM} Bennis and Mahdou, and independently Zhao and Huang in \cite[Thm. 3.9]{ZH}, studied relations between $m$-strongly Gorenstein projective modules and Gorenstein projective modules. Below in Theorem \ref{teo:equivmSGP} we generalize their results to $m$-periodic $(\mcA,\mcB)$-Gorenstein projective objects in terms of the resolution class $\mcA$ and the testing class $\mcB$.

\begin{figure}[H]
\centering
\includegraphics[width=0.55\textwidth]{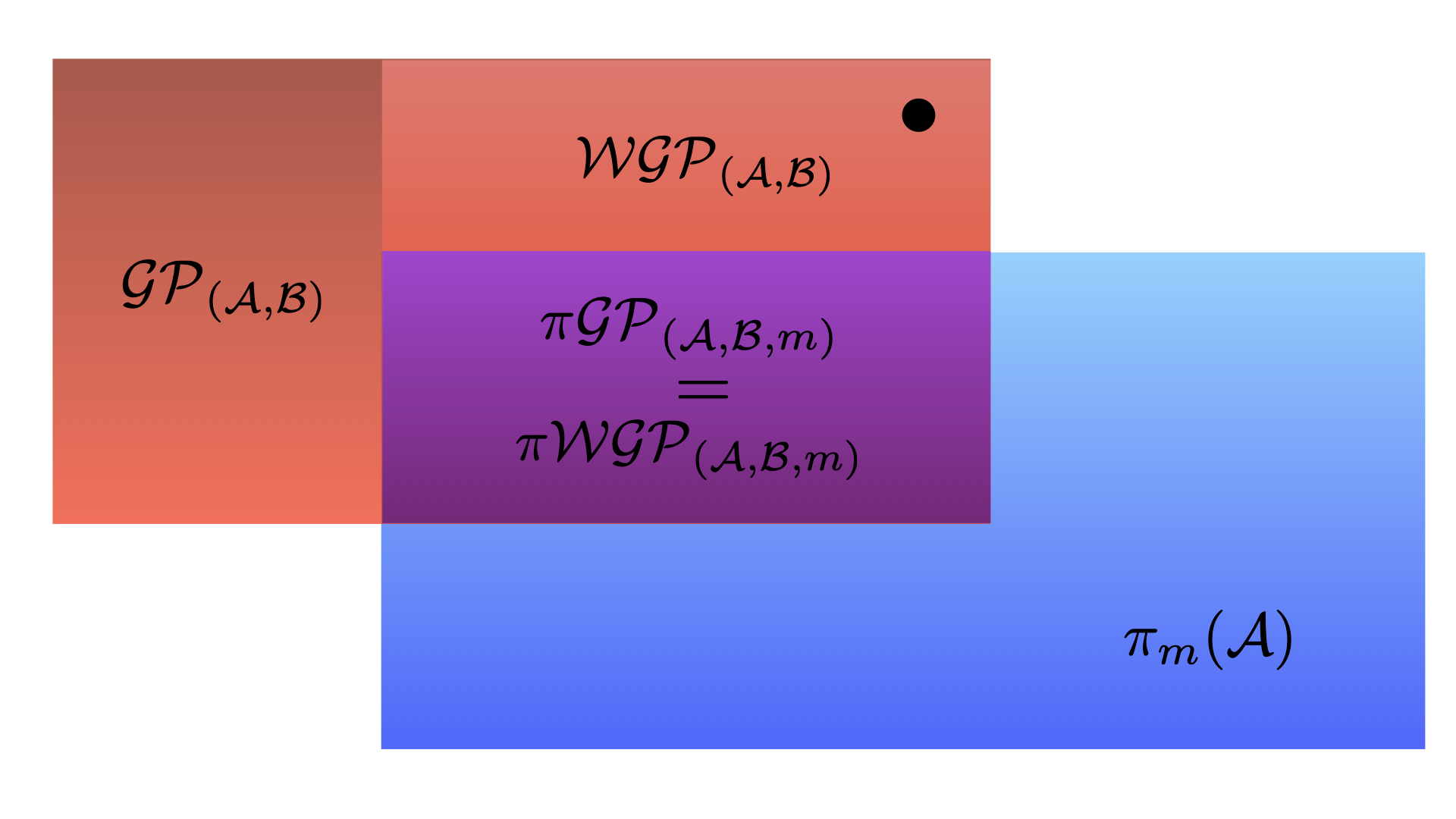}
\caption{A graphic depiction of the previous result. The dot represents the existence of an $(\mcA,\mcB)$-Gorenstein projective object which is not $m$-periodic for a certain $m$. See Example \ref{ex:sGP} for $m = 2$ and the simple projective object $P(3)$. The containment $\mathcal{GP}_{(\mcA,\mcB)} \subseteq \mathcal{WGP}_{(\mcA,\mcB)}$ from \cite[Prop. 3.14]{BMS} is also represented in the figure, and may be strict, as for instance in the case where $\mcA = \mcB = \text{projective objects}$ in an abelian category with not enough projective objects.}
\end{figure}
The following is the relative version of the well known fact that every Gorenstein projective module has either infinite or null projective dimension.

\begin{lemma}\label{lem:special_lemma}
Let $\mcA,\mcB \subseteq \mcC$ such that $\pd_{\mcB}(\mcA) = 0$, $\mcA$ is closed under extensions and direct summands, and $\omega:=\mcA\cap\mcB$ is a relative cogenerator in $\mcA$. Then, $\mathcal{WGP}_{(\mcA, \mcB)}$ is closed under extensions and $\mathcal{WGP}_{(\mcA,\mcB)} \cap \mcA^\wedge = \mcA$.
\end{lemma}

\begin{proof} 
The fact that $\mathcal{WGP}_{(\mcA, \mcB)}$ is closed under extensions is proved in  \cite[Thm. 3.30]{BMS}.\footnote{In this reference, it is assumed that $\mcA \cap \mcB$ is closed under finite coproducts. After a careful revision of the arguments there, one can note that this assumption is not necessary.} Regarding the equality $\mathcal{WGP}_{(\mcA,\mcB)} \cap \mcA^\wedge = \mcA$, the containment ($\supseteq$) is clear, since $\mcA$ is closed under extensions and direct summands (and so, under finite direct sums). Now let $M \in \mathcal{WGP}_{(\mcA,\mcB)} \cap \mcA^\wedge$. By \cite[Thm. 2.8]{BMS}, there exists a short exact sequence $K \rightarrowtail A \twoheadrightarrow M$ with $A \in \mcA$ and $K \in \omega^\wedge$. Moreover, by \cite[Coroll. 3.15 (b)]{BMS}, $\Ext^{\geq 1}(M,\omega^\wedge) = 0$, and so the previous sequence splits. It follows that $M \in \mcA$, since $\mcA$ is closed under direct summands.\footnote{The equality $\mathcal{WGP}_{(\mcA,\mcB)} \cap \mcA^\wedge = \mcA$ was originally proven in \cite[Coroll. 4.15 (c2)]{BMS} with the additional assumption that $\Ext^{\geq 1}(\mathcal{GP}_{(\mcA,\mcB)},\mcA) = 0$. As one can notice in our proof, this condition is not needed.}
\end{proof}

\begin{theorem}\label{teo:equivmSGP}
Let $\mcA,\mcB \subseteq \mcC$ such that $\pd_{\mcB}(\mcA) = 0$, and $\omega:=\mcA\cap \mcB$. Consider the following statements for any $M\in \mcC$:
\begin{enumerate}[(a)]
\item $M\in \pi\mathcal{GP}_{(\mathcal{A, B},m)}$.
\item There exists an $\mcA$-loop $A_\bullet$ at $M$ of length $m$ which is $\Hom(-,\mcB^{\wedge})$-acyclic.

\item There exists an $\mcA$-loop $A_\bullet$ at $M$ of length $m$, and an integer $i\geq 0$ such that  $M \in \bigcap^{m}_{k = 1} {}^{\perp_{k+i}}\mcB$ for any $B\in \mcB$ (equivalently, $B\in \mcB^{\wedge}$).

\item There exists an $\mcA$-loop (equivalently, an $\mcA^\wedge$-loop) $A_\bullet$ at $M$ of length $m$ such that $\bigoplus^{m}_{k = 1} Z_k(A_\bullet) \in \pi\mathcal{GP}_{({\rm free}(\mcA),\mcB,1)}$.

\item There exists an $\mcA$-loop (equivalently, an $\mcA^\wedge$-loop) $A_\bullet$ at $M$ of length $m$ such that $\bigoplus^{m}_{k = 1} Z_k(A_\bullet) \in \mathcal{WGP}_{({\rm free}(\mcA),\mcB)}$.
\end{enumerate}
Then, (a), (b) and (c) are equivalent, and the implications (a) $\Rightarrow$ (d) $\Rightarrow$ (e) hold true. Moreover, if $\mcA$ is closed under extensions and direct summands, and $\omega$ is a relative cogenerator in $\mcA$, then all conditions are equivalent. \\
\end{theorem}

\begin{proof} ~\
\begin{itemize}
\item (a) $\Rightarrow$ (b): Let $M \in \pi\mathcal{GP}_{(\mathcal{A, B},m)}$. By Proposition~\ref{pro:WSGP=SGP}, we have that $M \in \pi\mathcal{WGP}_{(\mathcal{A, B},m)}$, and there is an $\mcA$-loop $A_\bullet$ at $M$ of length $m$ with cycles in ${}^\perp\mcB$. Corollary \ref{cor:Li^perp=M^perp} shows that $Z_k(A_\bullet) \in {}^\perp\mcB^\wedge$ for every $k \in \mathbb{Z}$. The latter implies that $A_\bullet$ is $\Hom(-,\mcB^{\wedge})$-acyclic.

\item (b) $\Rightarrow$ (c) $\Rightarrow$ (a) follow by Corollaries \ref{cor:Li^perp=M^perp} and \ref{characterization2}.

\item (a) $\Rightarrow$ (d) follows as in \cite[Thm. 3.9]{ZH}.

\item (d) $\Rightarrow$ (e) is a consequence of Proposition \ref{prop:characterizationGPvsWGP}. For, note that $({\rm free}(\mcA),\mcB)$ is a hereditary pair. 
\end{itemize}
For the rest of the proof, assume that $\mcA$ is closed under extensions and direct summands, and that $\omega$ is a relative cogenerator in $\mcA$. Under these conditions, we have that ${\rm free}(\mcA) = \mcA$, and so $\mathcal{WGP}_{({\rm free}(\mcA), \mcB)}=\mathcal{WGP}_{(\mcA, \mcB)}$. Now let $A_\bullet$ be an $\mcA^\wedge$-loop at $M$ of length $m$ such that $\bigoplus^{m}_{k = 1} Z_k(A_\bullet) \in \mathcal{WGP}_{(\mcA,\mcB)}$. By Lemma~\ref{lem:special_lemma}, we know that $\mathcal{WGP}_{(\mcA,\mcB)}$ is closed under extensions and $\mathcal{WGP}_{(\mcA, \mcB)}\cap \mcA^{\wedge} = \mcA$. Thus, by taking the exact sequence 
\[
\bigoplus^{m}_{k = 1} Z_k(A_\bullet) \rightarrowtail \bigoplus^{m}_{k = 1} A_{k} \twoheadrightarrow \bigoplus^{m}_{k = 1} Z_k(A_\bullet)
\]
we get that $\bigoplus^{m}_{k = 1} A_{k} \in \mathcal{WGP}_{(\mcA, \mcB)} \cap \mcA^{\wedge} = \mcA$, and then $A_{k}\in \mcA$ since $\mcA$ is closed under direct summands. On the other hand,  $A_\bullet$ is $\Hom(-,\mcB)$-acyclic since it has cycles in $\mathcal{WGP}_{(\mcA,\mcB)} \subseteq {}^{\perp}\mcB$ (see \cite[Prop. 3.14]{BMS}).
\end{proof}

Some useful properties of relative $(\mcA,\mcB)$-Gorenstein objects are obtained in the case where the pair $(\mathcal{A,B})$ is \emph{GP-admissible} (resp., \emph{GI-admissible}) \cite[Defs. 3.1 \& 3.6]{BMS}. This means that:
\begin{itemize}
\item $\mcA$ and $\mcB$ are closed under finite coproducts and $\pd_{\mcB}(\mcA) = 0$.

\item Every object in $\mcC$ is the epimorphic image of an object in $\mcA$ (resp., every object in $\mcC$ can be embedded into an object in $\mcB$). 

\item $\mcA$ (resp., $\mcB$) is closed under extensions. 

\item $\mcA \cap \mcB$ is a relative cogenerator in $\mcA$ (resp., a relative generator in $\mcB$).
\end{itemize}

As we will see in Proposition~\ref{pro: GP sumandos directos de mGP}, $(\mcA,\mcB)$-Gorenstein projective objects are direct summands of $m$-periodic $(\mcA,\mcB)$-Gorenstein projective objects, provided that $(\mcA,\mcB)$ is GP-admissible. However, the class $\pi\mathcal{GP}_{(\mcA,\mcB, m)}$ may not be closed under direct summands in general. Indeed, a counterexample is given in \cite[Ex. 3.10]{ZH}. Let us extend this observation to the relative case with the following example.

\begin{example}\label{ex: piGPm->piGP1}
Consider the pair $(\mathcal{P}(\mcX), \mathcal{P}(\mcX))$ from Example~\ref{ex:sGP}-(1), which is not GP-admissible. It is clear that $\mathcal{P}(\mcX)$ is closed under extensions, finite coproducts and that $\Ext^{\geq 1}(\mathcal{P}(\mcX),\mathcal{P}(\mcX)) = 0$. From Lemma~\ref{lem:special_lemma}, we know that $\mathcal{WGP}_{\mathcal{P}(\mcX)}$ is closed under extensions and $\mathcal{P}(\mcX) = (\mathcal{P}(\mcX))^{\wedge} \cap \mathcal{WGP}_{\mathcal{P}(\mcX)}$. Therefore, all the conditions in Theorem~\ref{teo:equivmSGP} are equivalent.

On the other hand, notice that there are short exact sequences 
\[
S(1) \rightarrowtail P(2) \twoheadrightarrow S(2)\quad \mbox{and}\quad S(2) \rightarrowtail P(1) \twoheadrightarrow S(1)
\] 
which are $\Hom(-,\mathcal{P}(\mcX))$-acyclic (recall that $\mcX$ is a Frobenius subcategory of $\modu(\Lambda)$). This yields an exact and $\Hom(-,\mathcal{P}(\mcX))$-acyclic sequence 
\[
S(1) \rightarrowtail P(2) \to P(1) \twoheadrightarrow S(1),
\] 
and so $S(1) \in \pi\mathcal{GP}_{(\mathcal{P}(\mcX),\mathcal{P}(\mcX), 2)}$. Moreover, $S(1)\oplus S(2)\in \pi\mathcal{GP}_{(\mathcal{P}(\mcX),\mathcal{P}(\mcX),1)}$ by Theorem \ref{teo:equivmSGP}. However, $S(1) \notin \pi\mathcal{GP}_{(\mathcal{P}(\mcX),\mathcal{P}(\mcX), 1)}$. Indeed, if $S(1) \in \pi\mathcal{GP}_{(\mathcal{P}(\mcX),\mathcal{P}(\mcX), 1)}$ then there is an exact sequence $ S(1)\rightarrowtail P\twoheadrightarrow S(1)$ with $P\in \mathcal{P}(\mcX)$ which splits by the Auslander-Reiten formula \cite[IV.2 Lem. 2.12]{ASS}. Hence, $S(1) \in \mathcal{P}(\mcX)$ which is a contradiction. 
\end{example}


\section{Orthogonality relations and cluster tilting categories}\label{sec: orth relation}

Another important result by Zhao and Huang \cite[Prop. 3.7]{ZH} asserts that an $n$-strongly Gorenstein projective module $M$ is projective if, and only if, it is \emph{self-orthogonal}, that is, $\Ext^{\geq 1}(M,M) = 0$. We begin this section proving a relative version of this fact. To that end, we shall work with the following class of $m$-periodic $(\mcA, \mcB)$-Gorenstein projective objects.

\begin{definition}\label{def: piGPacyc}
An object $M \in \mcC$ is \textbf{proper $\bm{m}$-periodic $\bm{(\mcA,\mcB)}$-Gorenstein projective} if there exists an $\mcA$-loop at $M$ of length $m$ which is both $\Hom(\mcA, -)$-acyclic and $\Hom(-,\mcB)$-acyclic. \textbf{Proper $\bm{m}$-periodic $\bm{(\mcA,\mcB)}$-Gorenstein injective objects} are defined dually. The classes of proper $m$-periodic $(\mcA,\mcB)$-Gorenstein projective and injective objects in $\mcC$ will be denoted by $\pi\mathcal{GP}_{(\mcA,\mcB,m)}^{\rm ppr}$ and $\pi\mathcal{GI}_{(\mcA,\mcB,m)}^{\rm ppr}$, respectively.
\end{definition}

\begin{remark}\label{rem:acy} ~\
\begin{enumerate}
\item Clearly, the containment
\[
\pi\mathcal{GP}_{(\mcA,\mcB,m)}^{\rm ppr} \subseteq \pi\mathcal{GP}_{(\mcA,\mcB, m)}
\] 
is valid for any $m\geq 1$, and may by strict (see Example \ref{ex:sGP}-(2,3,4)). The equality holds for instance in the case where $\mcA \subseteq \mcP$. 

\item Proper $m$-periodic relative Gorenstein projective objects can be used to give a characterization of $(n+1)$-rigid subcategories. Indeed, note that if $\mcA$ is closed under finite coproducts, then $\mcA$ is $(n+1)$-rigid if, and only if, $\pi\mathcal{GP}^{\rm ppr}_{(\mcA,\mcB, 1)} \subseteq \bigcap^n_{i = 1} \mcA^{\perp_i}$, for any $\mcB \subseteq \mcC$.

\item If $M \in \pi\mathcal{GP}^{\rm ppr}_{(\mcA,\mcB,m)}$ and $A_\bullet$ is a $\Hom(\mcA,-)$-acyclic and $\Hom(-,\mcB)$-acyclic $\mcA$-loop at $M$ of length $m$, then $Z_k(A_\bullet) \in \pi\mathcal{GP}^{\rm ppr}_{(\mcA,\mcB,m)}$ for every $k \in \mathbb{Z}$.
\end{enumerate}
\end{remark}

Despite the fact that the class of (proper) relative $m$-periodic Gorenstein projective objects is not closed under direct summands in general, in some cases, this closure property holds for certain direct summands. Moreover, under specific circumstances we can give a characterization of when the class of (proper) relative $m$-periodic Gorenstein objects is closed under extensions. This is specified in the following result.

\begin{theorem}
Let $\mcA, \mcB$ be classes of objects in $\mcC$ such that $\Ext^1(\mathcal{A,B}) = 0$. Let $\omega := \mcA \cap \mcB$ be closed under finite coproducts and direct summands.
\begin{enumerate}
\item If $M = N \oplus W$ with $W \in \omega$, then $M \in \pi\mathcal{GP}^{\rm ppr}_{(\omega,\mcB,m)}$ if, and only if, $N \in \pi\mathcal{GP}^{\rm ppr}_{(\omega,\mcB,m)}$. 

\item The following assertions are equivalent for any short exact sequence
\begin{align}\label{eqn:sequence_closure}
X & \rightarrowtail Y \twoheadrightarrow Z.
\end{align}
\begin{enumerate}[(a)]
\item If $X, Z \in \pi\mathcal{GP}^{\rm ppr}_{(\omega,\mcB,m)}$, then $Y \in \pi\mathcal{GP}^{\rm ppr}_{(\omega,\mcB,m)}$. That is, $\pi\mathcal{GP}^{\rm ppr}_{(\omega,\mcB,m)}$ is closed under extensions.


\item If $Y, Z \in \pi\mathcal{GP}^{\rm ppr}_{(\omega,\mcB,m)}$, then $X \in \pi\mathcal{GP}^{\rm ppr}_{(\omega,\mcB,m)}$ if, and only if, $X \in \omega^{\perp_1}$. That is, $\pi\mathcal{GP}^{\rm ppr}_{(\omega,\mcB,m)}$ is closed under epikernels in $\omega^{\perp_1}$.

\item If $X, Y \in \pi\mathcal{GP}^{\rm ppr}_{(\omega,\mcB,m)}$, then $Z \in \pi\mathcal{GP}^{\rm ppr}_{(\omega,\mcB,m)}$ if, and only if, $Z \in {}^{\perp_1}\mcB$. That is, $\pi\mathcal{GP}^{\rm ppr}_{(\omega,\mcB,m)}$ is closed under monocokernels in ${}^{\perp_1}\mcB$. 
\end{enumerate}
\end{enumerate}
\end{theorem}

\begin{proof} ~\
\begin{enumerate}
\item For the ``if part'', if $N \in \pi\mathcal{GP}^{\rm ppr}_{(\omega,\mcB,m)}$ then $N \simeq Z_{mk}(W_\bullet)$ for every $k \in \mathbb{Z}$, where $W_\bullet$ is a $\Hom(\omega,-)$-acyclic and $\Hom(-,\mcB)$-acyclic $\omega$-loop at $N$ of length $m$. On the other hand, let $\overline{W}$ denote the complex with $W$ at degrees $1$ and $0$, and $0$ otherwise, where the only nonzero differential map is the identity on $W$. Consider its $(m-1)k$ suspension complex $\overline{W}[(m-1)k]$\footnote{Recall that the {\bf $\bm{m}$-th suspension} of a complex $X_\bullet$ is defined as the complex $X_\bullet[m]$ such that $(X_{\bullet}[m])_i := X_{i-m}$ and $\partial^{X_\bullet[m]}_i := (-1)^mX_{i-m}$, for every $i \in \mathbb{Z}$.} Then, we have that $W \simeq Z_{mk}(\bigoplus_{k \in \mathbb{Z}} \overline{W}[(m-1)k])$, and so 
\begin{align*}
N \oplus W & \simeq Z_{mk}(W_\bullet) \oplus Z_{mk}\left(\bigoplus_{k \in \mathbb{Z}} \overline{W}[(m-1)k]\right) \\
& \simeq Z_{mk}\left(W_\bullet \oplus \bigoplus_{k \in \mathbb{Z}} \overline{W}[(m-1)k] \right),
\end{align*}
for every $k \in \mathbb{Z}$, where $W_\bullet \oplus \left[ \bigoplus_{k \in \mathbb{Z}} \overline{W}[(m-1)k] \right]$ is a $\Hom(\omega,-)$-acyclic and $\Hom(-,\mcB)$-acyclic $\omega$-loop at $N \oplus W$ of length $m$. Hence, $M = N \oplus W \in \pi\mathcal{GP}^{\rm ppr}_{(\omega,\mcB,m)}$.

Now for the ``only if'' part, suppose that $M \in \pi\mathcal{GP}^{\rm ppr}_{(\omega,\mcB,m)}$. Then, we can consider a $\Hom(\omega,-)$-acyclic and $\Hom(-,\mcB)$-acyclic $\omega$-loop at $M$ of length $m$, say $W_\bullet$. These acyclicity conditions and the fact that $\Ext^1(\mathcal{A,B}) = 0$ imply that $Z_k(W_\bullet) \in \omega^{\perp_1} \cap {}^{\perp_1}\mcB$ for every $k \in \mathbb{Z}$. In particular, we have $N \in \omega^{\perp_1} \cap {}^{\perp_1}\mcB$. The pushout of $M \rightarrowtail W_m$ and $M \twoheadrightarrow N$ yields the following solid diagram:
\[
\begin{tikzpicture}[description/.style={fill=white,inner sep=2pt}] 
\matrix (m) [matrix of math nodes, row sep=2.3em, column sep=2.3em, text height=1.25ex, text depth=0.25ex] 
{ 
W & M & N \\ 
W & W_m & W'_m \\
{} & Z_m(W_\bullet) & Z_m(W_\bullet) \\
}; 
\path[->] 
(m-1-2)-- node[pos=0.5] {\footnotesize$\mbox{\bf po}$} (m-2-3) 
;
\path[>->]
(m-1-1) edge (m-1-2) 
(m-1-2) edge (m-2-2)
(m-2-1) edge (m-2-2) 
(m-1-3) edge (m-2-3)
;
\path[->>]
(m-1-2) edge (m-1-3) 
(m-2-2) edge (m-2-3) edge (m-3-2)
(m-2-3) edge (m-3-3)
;
\path[-,font=\scriptsize]
(m-1-1) edge [double, thick, double distance=2pt] (m-2-1)
(m-3-2) edge [double, thick, double distance=2pt] (m-3-3)
;
\end{tikzpicture} .
\]
Note in the right-hand side column that $N, Z_m(W_\bullet) \in {}^{\perp_1}\mcB$ implies that $W'_m \in {}^{\perp_1}\mcB$. So the middle row is split exact. Thus, $W_m \simeq W'_m \oplus W$, which in turn implies that $W'_m \in \omega$ since $\omega$ is closed under direct summands. So the right-hand side column is a $\Hom(\omega,-)$-acyclic and $\Hom(-,\mcB)$-acyclic short exact sequence with $W'_m \in \omega$. On the other hand, taking the pullback of $W_1 \twoheadrightarrow M$ and $N \rightarrowtail M$ yields the following solid diagram:
\[
\begin{tikzpicture}[description/.style={fill=white,inner sep=2pt}] 
\matrix (m) [matrix of math nodes, row sep=2.3em, column sep=2.3em, text height=1.25ex, text depth=0.25ex] 
{ 
Z_1(W_\bullet) & W'_1 & N \\ 
Z_1(W_\bullet) & W_1 & M \\
{} & W & W \\
}; 
\path[->] 
(m-1-2)-- node[pos=0.5] {\footnotesize$\mbox{\bf pb}$} (m-2-3) 
;
\path[>->]
(m-1-1) edge (m-1-2) 
(m-1-2) edge (m-2-2)
(m-2-1) edge (m-2-2) 
(m-1-3) edge (m-2-3)
;
\path[->>]
(m-1-2) edge (m-1-3) 
(m-2-2) edge (m-2-3) edge (m-3-2)
(m-2-3) edge (m-3-3)
;
\path[-,font=\scriptsize]
(m-1-1) edge [double, thick, double distance=2pt] (m-2-1)
(m-3-2) edge [double, thick, double distance=2pt] (m-3-3)
;
\end{tikzpicture} .
\]
As in the previous diagram, we can note that the top row is a $\Hom(\omega,-)$-acyclic and $\Hom(-,\mcB)$-acyclic short exact sequence with $W'_1 \in \omega$. Hence, the sequence $N \rightarrowtail W'_m \to W_{m-1} \to \cdots \to W_2 \to W'_1 \twoheadrightarrow N$ gives rise to a $\Hom(\omega,-)$-acyclic and $\Hom(-,\mcB)$-acyclic $\omega$-loop at $N$ of length $m$, and so $N \in \pi\mathcal{GP}^{\rm ppr}_{(\omega,\mcB,m)}$. 

\item We only prove the implications (a) $\Rightarrow$ (b) $\Rightarrow$ (a), since (a) $\Rightarrow$ (c) $\Rightarrow$ (a) follows in a similar way. 
\begin{itemize}
\item (a) $\Rightarrow$ (b): The ``only if'' part is clear since $\pi\mathcal{GP}^{\rm ppr}_{(\omega,\mcB,m)} \subseteq \omega^{\perp_1}$. Now suppose we are given a short exact sequence as \eqref{eqn:sequence_closure} with $Y, Z \in \pi\mathcal{GP}^{\rm ppr}_{(\omega,\mcB,m)}$ and $X \in \omega^{\perp_1}$. By Remark \ref{rem:acy}-(3), we can take a short exact sequence $Z' \rightarrowtail W \twoheadrightarrow Z$ with $W \in \omega$ and $Z' \in \pi\mathcal{GP}^{\rm ppr}_{(\omega,\mcB,m)}$. The pullback of $Y \twoheadrightarrow Z$ and $W \twoheadrightarrow Z$ yields the following solid diagram:
\[
\begin{tikzpicture}[description/.style={fill=white,inner sep=2pt}] 
\matrix (m) [matrix of math nodes, row sep=2.3em, column sep=2.3em, text height=1.25ex, text depth=0.25ex] 
{ 
{} & Z' & Z' \\ X & W' & W \\ X & Y & Z \\
}; 
\path[->] 
(m-2-2)-- node[pos=0.5] {\footnotesize$\mbox{\bf pb}$} (m-3-3) 
;
\path[>->]
(m-1-2) edge (m-2-2) 
(m-1-3) edge (m-2-3)
(m-2-1) edge (m-2-2) 
(m-3-1) edge (m-3-2)
;
\path[->>]
(m-2-2) edge (m-2-3) edge (m-3-2) 
(m-3-2) edge (m-3-3) 
(m-2-3) edge (m-3-3)
;
\path[-,font=\scriptsize]
(m-1-2) edge [double, thick, double distance=2pt] (m-1-3)
(m-2-1) edge [double, thick, double distance=2pt] (m-3-1)
;
\end{tikzpicture} .
\]
Since $X \in \omega^{\perp_1}$, the central row splits, and so $W' \simeq W \oplus X$. On the other hand, $W' \in \pi\mathcal{GP}^{\rm ppr}_{(\omega,\mcB,m)}$ by using (a) in the central column. Hence, part (1) implies that $X \in \pi\mathcal{GP}^{\rm ppr}_{(\omega,\mcB,m)}$. 

\item (b) $\Rightarrow$ (a): Suppose that $X, Z \in \pi\mathcal{GP}^{\rm ppr}_{(\omega,\mcB,m)}$ in \eqref{eqn:sequence_closure}. Then, $Y \in \omega^{\perp_1}$. On the other hand, consider a short exact sequence $X \rightarrowtail W \twoheadrightarrow X'$ with $W \in \omega$ and $X' \in \pi\mathcal{GP}^{\rm ppr}_{(\omega,\mcB,m)}$. Taking the pushout of $X \rightarrowtail W$ and $X \rightarrowtail Y$ yields the following solid diagram:
\[
\begin{tikzpicture}[description/.style={fill=white,inner sep=2pt}] 
\matrix (m) [matrix of math nodes, row sep=2.3em, column sep=2.3em, text height=1.25ex, text depth=0.25ex] 
{ 
X & Y & Z \\
W& W' & Z \\
X' & X' & {} \\
}; 
\path[->] 
(m-1-1)-- node[pos=0.5] {\footnotesize$\mbox{\bf po}$} (m-2-2) 
;
\path[>->]
(m-1-1) edge (m-1-2)
(m-2-1) edge (m-2-2)
(m-1-2) edge (m-2-2)
(m-1-1) edge (m-2-1)
;
\path[->>]
(m-2-2) edge (m-2-3) edge (m-3-2)
(m-2-1) edge (m-3-1)
(m-1-2) edge (m-1-3)
;
\path[-,font=\scriptsize]
(m-1-3) edge [double, thick, double distance=2pt] (m-2-3)
(m-3-1) edge [double, thick, double distance=2pt] (m-3-2)
;
\end{tikzpicture} .
\]
Since $Z \in \pi\mathcal{GP}^{\rm ppr}_{(\omega,\mcB,m)} \subseteq {}^{\perp_1}\mcB$, the central row splits, and so $W' \simeq W \oplus Z$. By part (1), $W' \in \pi\mathcal{GP}^{\rm ppr}_{(\omega,\mcB,m)}$. Then, applying (b) to the central column, we obtain that $Y \in \pi\mathcal{GP}^{\rm ppr}_{(\omega,\mcB,m)}$. 
\end{itemize}
\end{enumerate}
\end{proof}

\begin{remark}
With slightly different hypotheses, the previous result also holds for $m$-periodic $(\mcA,\mcB)$-Gorenstein projective objects. Indeed, if $\mcA$ and $\mcB$ are classes of objects in $\mcC$ with $\mcA$ closed under finite coproducts, direct summands and epikernels, satisfying $\Ext^1(\mcA, \mcB) = 0$, then $M\in \pi\mathcal{GP}_{(\mcA, \mcB, m)}$ if, and only if, $M\oplus W \in \pi\mathcal{GP}_{(\mcA, \mcB, m)}$ with $W \in \mcA\cap \mcB$. This follows as in \cite[Thm. 3.11]{ZH}. Moreover, this property for direct summands can be used to show that the following are equivalent:
\begin{enumerate}[(a)]
\item $\pi\mathcal{GP}_{(\mathcal{A,B},m)}$ is closed under extensions.

\item $\pi\mathcal{GP}_{(\mathcal{A,B},m)}$ is closed under epikernels in $\mcA^{\perp_1}$.

\item $\pi\mathcal{GP}_{(\mathcal{A,B},m)}$ is closed under monocokernels in ${}^{\perp_1}\mcB$.
\end{enumerate}
In summary, it is unlikely that the $m$-periodic $(\mcA,\mcB)$-Gorenstein projective objects form a class closed under direct summands or under extensions. 
\end{remark}

It is well known that the containment
\[
\{ \text{projective modules} \} \  \subseteq \ \{ \text{$n$-strongly Gorenstein projective modules} \}
\] 
may be strict (see Remark~\ref{rem:sGP}-(4)). However, one could be interested to look for objects in $\pi\mathcal{GP}^{\rm ppr}_{(\mcA,\mcB, m)}$ which belong to $\mcA$. This is specified in the following result, which generalizes \cite[Prop. 3.7]{ZH}.

\begin{proposition}\label{pro: (M,M)<m->A}
Let $1 \leq m \leq n$ and $\mcA \subseteq \mcC$ be closed under direct summands and $(n+1)$-rigid. If $M \in \pi\mathcal{GP}^{\rm ppr}_{(\mcA,\mcB, m)}$ and $\Ext^{\leq m}(M,M) = 0$, then $M \in \mcA$.
\end{proposition}

\begin{proof}
The case $m = 1$ is straightforward and does not require the assumption that $\Ext^{\leq n}(\mcA,\mcA) = 0$. So we may assume that $m > 1$. Now let $A_\bullet$ be a $\Hom(\mcA,-)$-acyclic and $\Hom(-,\mcB)$-acyclic $\mcA$-loop at $M$ of length $m$, and consider the short exact sequence $Z_1(A_\bullet) \rightarrowtail A_1 \twoheadrightarrow M$. Since $M \in \bigcap^n_{i = 1} \mcA^{\perp_i}$ by Proposition \ref{prop:ex}, and $\Ext^{\leq m}(M,M) = 0$ by the assumption, we obtain that 
\[
\Ext^{\leq m-1}(Z_1(A_\bullet), M) = 0.
\] 
The latter applied to the sequence $Z_2(A_\bullet) \rightarrowtail A_2 \twoheadrightarrow Z_1(A_\bullet)$ yields 
\[
\Ext^{\leq m-2}(Z_2(A_\bullet),M) = 0.
\] 
We can thus note inductively that 
\[
\Ext^{\leq m-j}(Z_j(A_\bullet),M) = 0 \text{ \ for every $1 \leq j < m$}. 
\]
In particular, $\Ext^1(Z_{m-1}(A_\bullet),M) = 0$, which implies that $M \rightarrowtail A_m \twoheadrightarrow Z_{m-1}(A_\bullet)$ is a split exact sequence, and hence $M \in \mcA$ since $\mcA$ is closed under direct summands. 
\end{proof}

\begin{corollary}\label{cor: M in A<-> (MM)<m}
Let $1 \leq m \leq n$ and $\mcA \subseteq \mcC$ be an $(n+1)$-rigid subcategory closed under direct summands. The following are equivalent for any $M \in \pi\mathcal{GP}^{\rm ppr}_{(\mcA,\mcB, m)}$:
\begin{enumerate}[(a)]
\item $M \in \mcA$.

\item $\Ext^{\leq n}(M,M) = 0$.

\item $\Ext^{\leq m}(M,M) = 0$.
\end{enumerate}
In particular, if $\mcA$ is self-orthogonal, then for any $M \in \pi\mathcal{GP}^{\rm ppr}_{(\mcA,\mcB, m)}$ one has that $M \in \mcA$ if, and only if, $M$ is self-orthogonal.
\end{corollary}

Below we give a weak version of Definition~\ref{def: piGPacyc} in order to obtain analogous outcomes of Proposition~\ref{pro: (M,M)<m->A} and Corollary~\ref{cor: M in A<-> (MM)<m}. We only present the statements of such results without proofs.

\begin{definition}\label{def: piWGPacyc}
An object $M \in \mcC$ is \textbf{proper weakly $\bm{m}$-periodic $\bm{(\mcA,\mcB)}$-Gorenstein projective} if there exists an $\mcA$-loop at $M$ of length $m$ with cycles in $\mcA{}^{\perp}\cap{}^{\perp}\mcB$. Dually, we have the notion of \textbf{proper weakly $\bm{m}$-periodic $\bm{(\mcA,\mcB)}$-Gorenstein injective objects} in $\mcC$. These classes of objects will be denoted by $\pi\mathcal{WGP}_{(\mcA,\mcB, m)}^{\rm ppr}$ and $\pi\mathcal{WGI}_{(\mcA,\mcB, m)}^{\rm ppr}$, respectively. 
\end{definition}

\begin{proposition}\label{prop:weakM in A<-> (MM)<m}
Let $1 \leq m \leq n$ and $\mcA \subseteq \mcC$ be an $(n+1)$-rigid subcategory closed under direct summands. The following assertions hold true:
\begin{enumerate}
\item If $M \in \pi\mathcal{WGP}^{\rm ppr}_{(\mcA,\mcB, m)}$ and $\Ext^{\leq m}(M,M) = 0$, then $M \in \mcA$.

\item The following are equivalent for any $M \in \pi\mathcal{WGP}^{\rm ppr}_{(\mcA,\mcB, m)}$:
\begin{enumerate}[(a)]
\item $M \in \mcA$.

\item $\Ext^{\leq n}(M,M) = 0$.

\item $\Ext^{\leq m}(M,M) = 0$.
\end{enumerate}
In particular, if $\mcA$ is self-orthogonal, then for any $M \in \pi\mathcal{WGP}^{\rm ppr}_{(\mcA,\mcB, m)}$ one has that $M \in \mcA$ if, and only if, $M$ is self-orthogonal.
\end{enumerate}
\end{proposition}

As we know from the beginning of Section \ref{sec:acyclicity}, $n$-cluster tilting subcategories provide a wide source of examples of $(n+1)$-rigid subcategories. In the rest of this section, we give some applications related to a generalization of this concept given in Argud\'in and Mendoza's \cite[Def. 2.8]{monroy2021relative}. Specifically, given $n \geq 1$ and a class of objects $\mcX \subseteq \mcC$, we say that $\mathcal{T} \subseteq \mcC$ is \emph{$(n+1)$-$\mathcal{X}$-cluster tilting} in $\mcC$ if the following hold:
\begin{enumerate}
\item $\mathcal{T} = \add(\mathcal{T})$; 

\item There exists $\alpha \subseteq \mathcal{X}^{\perp} \cap \mathcal{T}^{\perp}$, which is a relative cogenerator in $\mathcal{X}$;

\item There exists $\beta \subseteq {}^{\perp}\mathcal{X} \cap {}^{\perp}\mathcal{T}$, which is a relative generator in $\mathcal{X}$;

\item $\mathcal{X}$ is functorially finite;

\item $\mathcal{X}\cap (\bigcap_{i=1}^{n}{}^{\perp_{i}}\mathcal{T}) = \mathcal{T} = \mathcal{X} \cap  (\bigcap_{i=1}^{n}\mathcal{T}^{\perp_{i}})$.
\end{enumerate}



As for (weak) $m$-periodic $(\mcA,\mcB)$-Gorenstein objects, it is clear that 
\[
\pi\mathcal{WGP}_{(\mcA,\mcB, m)}^{\rm ppr} \subseteq \pi\mathcal{GP}_{(\mcA,\mcB, m)}^{\rm ppr}.
\] 
The converse containment is not true en general, although it can be characterized for certain choices of $\mcA$ and $\mcB$, as the following result shows.

\begin{theorem}
Let $\mathcal{T} \subseteq \mcC$ be an $(n+1)$-$\mathcal{X}$-cluster tilting subcategory in $\mcC$ with $n\geq 1$. The following assertions hold true:
\begin{enumerate}
\item For any $m \geq 1$, the following equalities hold 
\[
\mathcal{X}\cap \pi\mathcal{GP}_{(\mathcal{T},\mathcal{T}, m)}^{\rm ppr}=\mathcal{T}=\mathcal{X}\cap \pi\mathcal{GP}_{(\mathcal{T},\mathcal{T}, m)}.
\]

\item If in addition $\mathcal{X}$ is closed under epikernels, the following are equivalent: 
\begin{enumerate}[(a)]
\item $\pi\mathcal{WGP}_{(\mathcal{T,T },m)} = \pi\mathcal{GP}_{(\mathcal{T, T},m)}$ for any $m\geq 1$;

\item $\Ext^{\geq 1}(\mathcal{T, T}) = 0$;

\item $\mathcal{T} = \add(\alpha) = \add(\beta)$;

\item $\mathcal{T} = \mathcal{X}$.
\end{enumerate}
\end{enumerate}
\end{theorem}

\begin{figure}[H]
\centering
\includegraphics[width=0.55\textwidth]{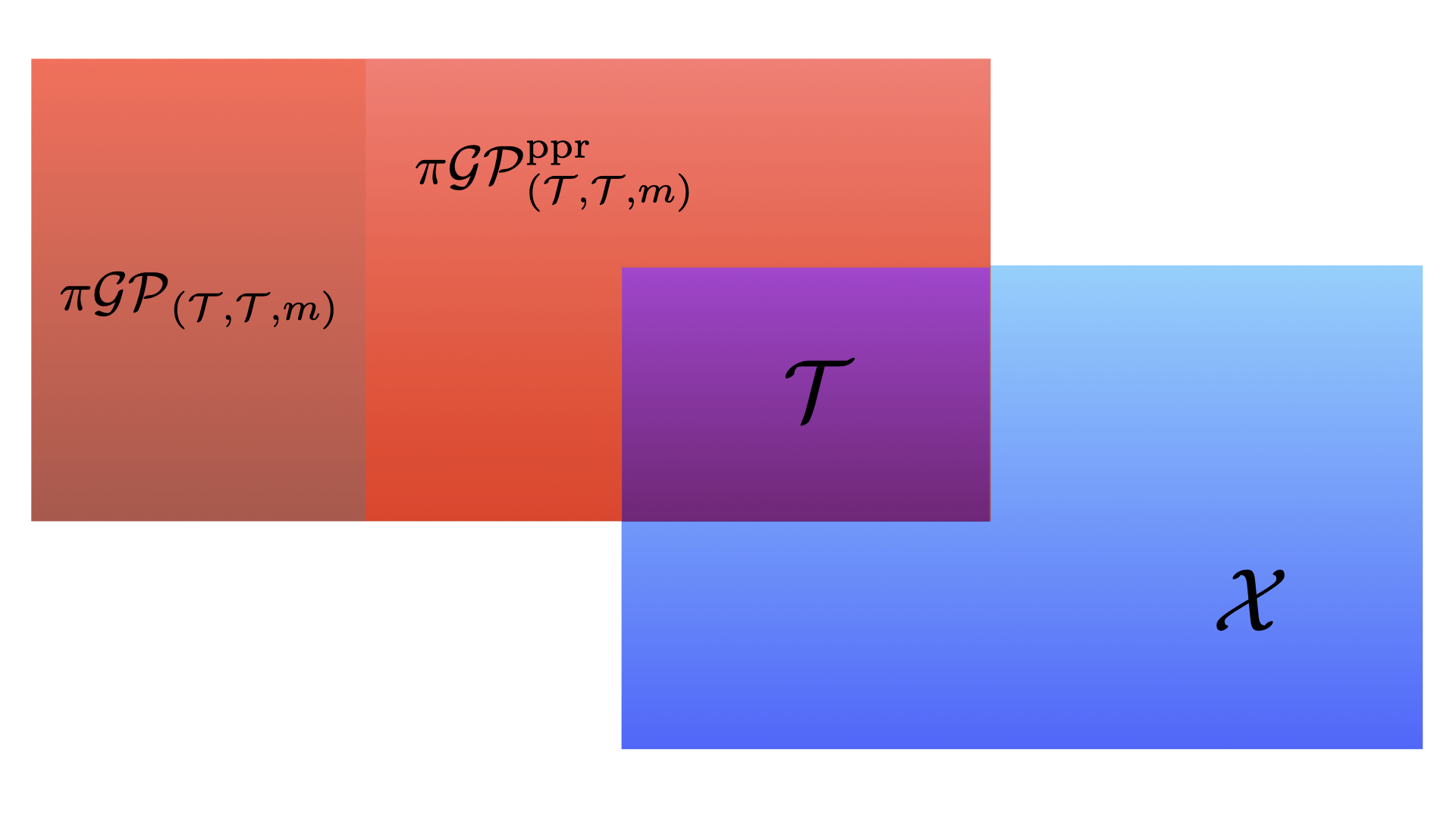}
\caption{Relations between an $(n+1)$-$\mathcal{X}$-cluster tilting subcategory $\mathcal{T}$ and (proper) $m$-periodic Gorenstein projective objects relative to $\mathcal{T}$. }
\end{figure}

\begin{proof} ~\
\begin{enumerate}
\item First, note from Proposition~\ref{prop:ex} that
\[
\pi\mathcal{GP}_{(\mathcal{T},\mathcal{T},m)}^{\rm ppr}\subseteq \Big( \bigcap_{j=1}^n \mathcal{T}^{\perp_{j}} \Big) \cap \pi\mathcal{GP}_{(\mathcal{T},\mathcal{T},m)}.
\] 
Thus, by taking the intersection with $\mathcal{X}$ and the definition of $(n+1)$-$\mathcal{X}$-cluster tilting class, we obtain that
\[
\mathcal{X}\cap\pi\mathcal{GP}_{(\mathcal{T},\mathcal{T},m)}^{\rm ppr}\subseteq \mathcal{X}\cap \Big( \bigcap_{j=1}^n \mathcal{T}^{\perp_{j}} \Big) \cap \pi\mathcal{GP}_{(\mathcal{T},\mathcal{T},m)} = \mathcal{T}\cap \pi\mathcal{GP}_{(\mathcal{T},\mathcal{T},m)}\subseteq \mathcal{T}
\]
while the converse containment is clear. Hence, the first equality holds true. The proof of the second one is analogous by using the dual of Proposition~\ref{prop:ex} instead.

\item The equivalence (a) $\Leftrightarrow$ (b) follows from Proposition~\ref{pro:WSGP=SGP}. 

Regarding (b) $\Leftrightarrow$ (c), the converse implication is clear by the definition of $(n+1)$-$\mathcal{X}$-cluster tilting subcategories. For (b) $\Rightarrow$ (c), condition (b) implies that $\Ext^{\geq 1}(\mathcal{T}, \mathcal{T}^{\wedge}_{n}) = 0$ by \cite[Prop. 2.6]{HMP}. Moreover, under the assumption that $\mcX$ is closed under epikernels, one can show as in \cite[Prop. 5.26]{HMP}, that $\mcX \subseteq \mathcal{T}^{\wedge}_{n}$. Thus, $\Ext^{\geq 1}(\mathcal{T}, \mcX) = 0$. Now, let $T \in \mathcal{T}$. By conditions (3) and (5) in the definition of $(n+1)$-$\mathcal{X}$-cluster tilting subcategories, there is an exact sequence $X \rightarrowtail B \twoheadrightarrow T$ with $B \in \beta$ and $X \in \mcX$. Since $T \in {}^{\perp_{1}}\mcX$ this sequence splits, and so $T \in \add (\beta)$. Thus, $\mathcal{T} \subseteq \add (\beta)$. The converse containment, on the other hand, holds since $\mathcal{T} = \add(\mathcal{T})$ and $\beta \subseteq \mathcal{T}$ (see \cite[Coroll. 2.10]{monroy2021relative}). In a similar way, we can prove $\mathcal{T} = \add(\alpha)$.

The assumption $\mathcal{T}=\add (\beta)$ in condition (c) also implies that 
\[
\Ext^{\geq 1}(\mathcal{T}, \mcX) = \Ext^{\geq 1}(\add (\beta), \mcX) = 0.
\] 
Thus, we get (d) by \cite[Coroll. 2.10]{monroy2021relative}. Finally, for (d) $\Rightarrow$ (c) we only prove the equality $\mathcal{T} = \add(\beta)$. Notice first that since $\mathcal{T} = \mcX \cap (\bigcap_{i=1}^{n}{}^{\perp_{i}}\mathcal{T})$ and $\mathcal{T} = \mcX$, we get $\Ext^1(\mcX, \mcX) = 0$. Now, we see $\mcX\subseteq \add(\beta)$. Indeed, let $M \in \mcX$. By using that $\beta$ is a relative generator in $\mcX$, there exists a short exact sequence $X' \rightarrowtail B' \twoheadrightarrow M$ with $B' \in \beta$ and $X' \in \mcX$, which is split since $\Ext^{1}(\mcX, \mcX) = 0$. Hence, $M \in \add (\beta)$. The remaining containment follows from \cite[Coroll. 2.10]{monroy2021relative}. 
\end{enumerate}
\end{proof}

In the particular case of $(n+1)$-cluster tilting subcategories, we get:

\begin{corollary}\label{cor: SWGPclustertilting}
Let $\mcC$ be an abelian category with enough projective and injective objects, and $\mcD \subseteq \mcC$ be an $(n+1)$-cluster tilting subcategory in $\mcC$. The following hold true:
\begin{enumerate}
\item $\pi\mathcal{GP}^{\rm ppr}_{(\mathcal{D,D},m)} = \mathcal{D} = \pi\mathcal{GP}_{(\mathcal{D,D},m)}$ for every $m \geq 1$.

\item The following are equivalent:
\begin{enumerate}[(a)]
\item $\pi\mathcal{WGP}_{(\mathcal{D, D},m)} = \pi\mathcal{GP}_{(\mathcal{D, D},m)}$ for any $m\geq 1$;

\item $\Ext^{\geq 1}(\mathcal{D, D}) = 0$;

\item $\mcD = \mathcal{P}(\mcC) = \mathcal{I}(\mcC)$;

\item $\mcC = \mcD$.
\end{enumerate}
In particular, if one of the above conditions holds, then every short exact sequence in $\mcC$ is split.
\end{enumerate}
\end{corollary}


\section{Intersection properties}\label{sec: intersections}

Another important result of Zhao and Huang has to do with intersections of classes of $m$-strongly Gorenstein projective modules. Specifically, \cite[Thm. 3.5]{ZH} asserts that the intersection of the classes of $m$-strongly and $n$-strongly Gorenstein projective modules yields the class of $(m,n)$-strongly Gorenstein projective modules, where $(m,n)$ denotes the \emph{greatest common divisor} of $m$ and $n$. The purpose of this section is to generalize this result to the relative case.

\begin{theorem}\label{thm: MCD(m,n)}
Let $m, n\geq 1$ and $\mcA \subseteq \mcC$ be a class of objects in $\mcC$ closed under extensions and epikernels, with $0 \in \mcA$ and such that $\Ext^{1}(\mathcal{A, A}) = 0$. Let $\mcB \subseteq \mcC$ be another class satisfying $\Ext^{1}(\mathcal{A, B}) = 0$. Then,
\[
\pi\mathcal{GP}^{\rm ppr}_{(\mcA,\mcB,n)}\cap \pi\mathcal{GP}^{\rm ppr}_{(\mcA,\mcB,m)} =  \pi\mathcal{GP}^{\rm ppr}_{(\mcA,\mcB, (m,n))}.
\]
\end{theorem}

\begin{remark}\label{n mid m}
In the case $(m,n) = \min\{ m, n\}$, the assumptions that $\mcA$ is closed under extensions and $\Ext^{1}(\mathcal{A, B}) = 0 = \Ext^{1}(\mathcal{A, A})$ are not needed. 
\end{remark}

\begin{proof}
First, note that since $(m,n)$ divides $n$ and $m$, the containment $(\supseteq)$ is clear. So we just focus on proving $(\subseteq)$. Without loss of generality, and by the previous remark, we may assume that $\min\{ m, n\} = n > (m,n)$. Thus, we have that $n$ does not divide $m$, and so there exist unique positive integers $q_{1}$ and $r_{1}$ such that 
\[
m = q_{1} n + r_{1} \text{ \ where \ } 0 < r_1 < n.
\] 
\begin{itemize}
\item \textbf{Claim:} 
\begin{align}
\pi\mathcal{GP}^{\rm ppr}_{(\mcA,\mcB,n)} \cap \pi\mathcal{GP}^{\rm ppr}_{(\mcA,\mcB,m)} & \subseteq \pi\mathcal{GP}^{\rm ppr}_{(\mcA,\mcB,r_{1})}. \label{eqn:claim1}
\end{align}

Note first that, since $\mcA$ is closed under extensions and $0 \in \mcA$, then it is closed under finite coproducts. Let $M \in \pi\mathcal{GP}^{\rm ppr}_{(\mcA,\mcB,n)}\cap \pi\mathcal{GP}^{\rm ppr}_{(\mcA,\mcB,m)}$ and choose a $\Hom(\mcA,-)$-acyclic and $\Hom(-,\mcB)$-acyclic $\mcA$-loop at $M$ of length $m$, say $A_\bullet$. From Lemma~\ref{lem: Schanuel}, there exist $A, A' \in \mcA$ such that $M \oplus A \simeq Z_{q_1 n}(A_\bullet) \oplus A'$. Now, by taking the pullback of 
\[
Z_{q_1 n+1}(A_\bullet) \rightarrowtail A_{q_1 n + 1} \twoheadrightarrow Z_{q_1 n}(A_{\bullet}) \text{ \ and \ } A' \rightarrowtail M\oplus A \twoheadrightarrow Z_{q_1 n}(A_{\bullet})
\] 
we get the following solid diagram:
\[
\begin{tikzpicture}[description/.style={fill=white,inner sep=2pt}] 
\matrix (m) [matrix of math nodes, row sep=2.5em, column sep=2.5em, text height=1.25ex, text depth=0.25ex] 
{ 
{} & A' & A' \\
Z_{q_1 n+1}(A_\bullet) & X & M\oplus A \\
Z_{q_1 n+1}(A_\bullet) & A_{q_1 n + 1} & Z_{q_1 n}(A_{\bullet}) \\
}; 
\path[->] 
(m-2-2)-- node[pos=0.5] {\footnotesize$\mbox{\bf pb}$} (m-3-3) 
;
\path[>->]
(m-1-2) edge (m-2-2) (m-1-3) edge (m-2-3)
(m-2-1) edge (m-2-2) (m-3-1) edge (m-3-2)
;
\path[->>]
(m-2-2) edge (m-3-2) (m-2-3) edge (m-3-3)
(m-2-2) edge (m-2-3)
(m-3-2) edge (m-3-3)
;
\path[-,font=\scriptsize]
(m-2-1) edge [double, thick, double distance=2pt] (m-3-1)
(m-1-2) edge [double, thick, double distance=2pt] (m-1-3)
;
\end{tikzpicture} .
\]
Notice that $X\in \mcA$ since $\mcA$ is closed under extensions. Now, by considering the pullback of
\[
Z_{q_1 n + 1}(A_\bullet) \rightarrowtail X \twoheadrightarrow M\oplus A \text{ \ and \ } M \rightarrowtail M\oplus A \twoheadrightarrow A
\] 
we have the following solid diagram
\[
\parbox{2.5in}{
\begin{tikzpicture}[description/.style={fill=white,inner sep=2pt}] 
\matrix (m) [ampersand replacement=\&, matrix of math nodes, row sep=2.5em, column sep=2.5em, text height=1.25ex, text depth=0.25ex] 
{ 
Z_{q_1 n+1}(A_\bullet) \& Y \& M \\
Z_{q_1 n+1}(A_\bullet) \& X \& M \oplus A \\
{} \& A \& A \\
}; 
\path[->] 
(m-1-2)-- node[pos=0.5] {\footnotesize$\mbox{\bf pb}$} (m-2-3)
; 
\path[>->]
(m-2-1) edge (m-2-2) (m-1-3) edge (m-2-3)
(m-1-1) edge (m-1-2) (m-1-2) edge (m-2-2)
;
\path[->>]
(m-2-2) edge (m-2-3) (m-2-3) edge (m-3-3)
(m-2-2) edge (m-3-2) (m-1-2) edge (m-1-3)
;
\path[-,font=\scriptsize]
(m-1-1) edge [double, thick, double distance=2pt] (m-2-1)
(m-3-2) edge [double, thick, double distance=2pt] (m-3-3)
;
\end{tikzpicture} 
}
\]
where $Y\in \mcA$ since $\mcA$ is closed under epikernels. Moreover, under the assumptions that $\Ext^{1}(\mathcal{A, B}) = 0 = \Ext^{1}(\mathcal{A, A})$, and that $A_\bullet \in \Ch(\mcA)$ is an exact $\Hom(\mcA,-)$-acyclic and $\Hom(-,\mcB)$-acyclic complex, we get that $M \in {}^{\perp_{1}}\mcB$ and $Z_{q_1 n + 1}(A_{\bullet}) \in \mcA^{\perp_{1}}$, and so $Z_{q_1 n+1}(A_\bullet) \rightarrowtail Y \twoheadrightarrow M$ is an exact $\Hom(\mcA,-)$-acyclic and $\Hom(-,\mcB)$-acyclic short exact sequence. Hence, we have that the exact complex
\[
M \stackrel{f_{m+1}}\rightarrowtail A_m \to \cdots \to A_{q_1 n + 2} \to Y \stackrel{f_1}\twoheadrightarrow M
\]
is $\Hom(\mcA,-)$-acyclic and $\Hom(-,\mcB)$-acyclic, which in turn implies that $M \in \pi\mathcal{GP}^{\rm ppr}_{(\mcA,\mcB,r_{1})}$. \\
\end{itemize}

Continuing with the proof, the case $(n, r_1) = r_1$ follows by Remark~\ref{n mid m}. So, we can assume $(n, r_1) < r_1$. Thus, there exist unique positive integers $q_2$ and $r_2$ such that 
\[
n = q_2 r_1 + r_2 \text{ \ where \ } 0 < r_2 < r_1.
\] 
In a similar way as in the previous claim, one can prove that 
\begin{align}
\pi\mathcal{GP}^{\rm ppr}_{(\mcA,\mcB,n)} \cap \pi\mathcal{GP}^{\rm ppr}_{(\mcA,\mcB,r_1)} & \subseteq \pi\mathcal{GP}^{\rm ppr}_{(\mcA,\mcB,r_2)}. \label{eqn:claim2}
\end{align}
Then, from \eqref{eqn:claim1} and \eqref{eqn:claim2} we obtain
\[
\pi\mathcal{GP}^{\rm ppr}_{(\mcA,\mcB,n)}\cap \pi\mathcal{GP}^{\rm ppr}_{(\mcA,\mcB,m)}\subseteq
\pi\mathcal{GP}^{\rm ppr}_{(\mcA,\mcB,n)}\cap \pi\mathcal{GP}^{\rm ppr}_{(\mcA,\mcB, r_{1})} 
\subseteq \pi\mathcal{GP}^{\rm ppr}_{(\mcA,\mcB,r_{2})}.
\]
By continuing with the Euclidean algorithm, there exists $t\geq 1$ such that 
\[
r_t = q_{t+2} r_{t+1} \text{ \ and \  } r_{t+1} = (m, n),
\] 
and then
\[
\pi\mathcal{GP}^{\rm ppr}_{(\mcA,\mcB,r_{t})}\cap \pi\mathcal{GP}^{\rm ppr}_{(\mcA,\mcB, r_{t+1})}
= \pi\mathcal{GP}^{\rm ppr}_{(\mcA,\mcB, (m, n))}.
\]
Therefore, 
\[
\pi\mathcal{GP}^{\rm ppr}_{(\mcA,\mcB,n)}\cap \pi\mathcal{GP}^{\rm ppr}_{(\mcA,\mcB,m)}\subseteq
\pi\mathcal{GP}^{\rm ppr}_{(\mcA,\mcB,n)}\cap \pi\mathcal{GP}^{\rm ppr}_{(\mcA,\mcB,r_{1})}\subseteq 
\cdots\subseteq \pi\mathcal{GP}^{\rm ppr}_{(\mcA,\mcB, (m, n))}.
\]
\end{proof}

\begin{corollary}\label{cor: n+1cap n}
Let $\mcA \subseteq \mcC$ be a class of objects in $\mcC$ closed under extensions and epikernels, with $0 \in \mcA$ and such that $\Ext^{1}(\mathcal{A, A}) = 0$. Let $\mcB \subseteq \mcC$ be another class satisfying $\Ext^{1}(\mathcal{A, B}) = 0$. Then, 
\[
\pi\mathcal{GP}^{\rm ppr}_{(\mcA,\mcB,m+1)} \cap \pi\mathcal{GP}^{\rm ppr}_{(\mcA,\mcB,m)} =  \pi\mathcal{GP}^{\rm ppr}_{(\mcA,\mcB,1)}.
\]
In particular, 
\[
\bigcap_{m\geq 1}\pi\mathcal{GP}^{\rm ppr}_{(\mcA,\mcB,m)} = \pi\mathcal{GP}^{\rm ppr}_{(\mcA,\mcB,1)}.
\]
\end{corollary}

\begin{remark}
Notice that the converse implication of Corollary~\ref{cor: n+1cap n} does not necessarily hold true. For every $(n+1)$-cluster tilting subcategory $\mcD$ of $\mcC$, the pair $(\mcA,\mcB) := (\mcD,\mcD)$ satisfies the equality $\Ext^{1}(\mathcal{D,D}) = 0$ in Corollary~\ref{cor: n+1cap n}. However, $\mcD$ may not be closed under epikernels. In fact, if $\mcC$ is an abelian category with enough projective and injective objects, then $\mcD$ is closed under epikernels if, and only if, $\mathcal{D} = \mathcal{P}(\mathcal{C})$ \cite[Rmk. 5.28]{HMP}.
\end{remark}


\section{Relations with GP-admissible pairs}\label{sec: admissible}

The results given in Sections~\ref{sec: hereditary} and~\ref{sec: orth relation} on $m$-periodic Gorenstein objects relative to $(\mcA,\mcB)$ are obtained under the assumption that $\pd_{\mcB}(\mcA) = 0$. This section is devoted to provide more results in the contexts where $(\mcA, \mcB)$ is a GP-admissible pair or an $n$-cotorsion pair. In the sequel, we will frequently refer to \cite{BMS,HMP}.

\begin{proposition}\label{GP=SGP}
Let $(\mcA,\mcB)$ be a GP-admissible pair in $\mcC$. The following equalities hold:
\[
\pi\mathcal{GP}_{(\mathcal{GP}_{(\mcA,\mcB)},\mcB,m)} = \mathcal{GP}_{(\mcA,\mcB)}=\mathcal{GP}_{(\mcA,\mcB^\wedge)} = \pi\mathcal{GP}_{(\mathcal{GP}_{(\mcA,\mcB)},\mcB^\wedge,m)}.
\]
\end{proposition}

\begin{proof} 
From \cite[Thm. 3.32]{BMS} we know that $(\mathcal{GP}_{(\mcA,\mcB)},\mcB)$ is a GP-admissible pair and $\mathcal{GP}^{2}_{(\mcA,\mcB)} := \mathcal{GP}_{(\mathcal{GP}_{(\mcA,\mcB)},\mcB)} = 
\mathcal{GP}_{(\mcA,\mcB)}$. So,
\[
\mathcal{GP}_{(\mcA,\mcB)} \subseteq \pi\mathcal{GP}_{(\mathcal{GP}_{(\mcA,\mcB)},\mcB,m)} \subseteq \mathcal{GP}_{(\mathcal{GP}_{(\mcA,\mcB)},\mcB)} = \mathcal{GP}_{(\mcA,\mcB)},
\]
and then the first equality holds true. For the second equality, from \cite[Thm. 3.34 (a), Rmk. 3.13 \& Prop. 3.16]{BMS}, we get that $(\mcA,\mcB^\wedge)$ is a GP-admissible pair and that $\mathcal{GP}_{(\mcA,\mcB)} = \mathcal{GP}_{(\mcA,\mcB^\wedge)}$. Finally, the third equality follows by using the previous arguments to the GP-admissible pair $(\mcA,\mcB^\wedge)$.
\end{proof}

\begin{example}
For $\mcA = \mcB := \mathcal{P}(R)$, we get $\mathcal{GP}(R) = \pi\mathcal{GP}_{(\mathcal{GP}(R), \mathcal{P}(R),m)}$.
\end{example}

A well known result of Bennis and Mahdou's \cite{BM2} on the relation between Gorenstein projective and strongly Gorenstein projective modules asserts that a module is Gorenstein projective if, and only if, it is a direct summand of a strongly Gorenstein projective module. Below we state and prove the relative version of this result for $m$-periodic $(\mathcal{A,B})$-Gorenstein projective objects.

\begin{proposition}\label{pro: GP sumandos directos de mGP}
Let $(\mcA,\mcB)$ be a GP-admissible pair in an AB4 abelian category $\mcC$, with $\mcA$ closed under coproducts. Then, an object in $\mcC$ is $(\mcA,\mcB)$-Gorenstein projective if, and only if, it is a direct summand of a $1$-periodic $(\mcA,\mcB)$-Gorenstein projective object. In other words,
\[
\mathcal{GP}_{(\mcA,\mcB)} = {\rm add}(\pi\mathcal{GP}_{(\mcA,\mcB,1)}).
\]
\end{proposition}

The following proof replicates the arguments from \cite[Thm. 2.7]{BM2}.

\begin{proof}
The ``if'' part follows from the facts that $\pi\mathcal{GP}_{(\mcA,\mcB,1)} \subseteq \mathcal{GP}_{(\mcA,\mcB)}$ and that $\mathcal{GP}_{(\mcA,\mcB)}$ is closed under direct summands (since $(\mcA,\mcB)$ is a GP-admissible pair \cite[Coroll. 3.33]{BMS}). For the ``only if'' part, suppose we are given an object $C \in \mathcal{GP}_{(\mcA,\mcB)}$. Then, $C \simeq Z_0(A_\bullet)$ for some exact and $\Hom(-,\mcB)$-acyclic complex $A_\bullet \in \Ch(\mcA)$. Set
\[
N := \bigoplus_{k \in \mathbb{Z}} Z_k(A_\bullet).
\] 
Then, it is clear that $C$ is a direct summand of $N$. We show that $N$ is $1$-periodic $(\mcA,\mcB)$-Gorenstein projective. Consider the suspension complexes 
\[
A_\bullet[-k] = \cdots \to (A_\bullet[-k])_i \xrightarrow{(-1)^{-k}\partial^{A_\bullet}_{i+k}} (A_\bullet[-k])_{i-1}  \to \cdots.
\] 
Using the fact that $\mcC$ is AB4, we can form the exact complex 
\[
A'_\bullet := \bigoplus_{k \in \mathbb{Z}} A_\bullet[-k],
\] 
which belongs to $\Ch(\mcA)$ since $\mcA$ is closed under coproducts. On the other hand, for every $B \in \mcB$ we have that 
\[
\Hom(A'_\bullet,B) \cong \prod_{k \in \mathbb{Z}} \Hom(A_\bullet[-k],B),
\] 
where each $\Hom(A_\bullet[-k],B)$ is exact, and hence so is $\Hom(A'_\bullet,B)$. Note also that $N \simeq Z_i(A'_\bullet)$ for every $i \in \mathbb{Z}$, since $Z_k(A_\bullet) \simeq Z_0(A_\bullet[-k])$. Thus we have a $\Hom(-,\mcB)$-acyclic short exact sequence $N \rightarrowtail \bigoplus_{k \in \mathbb{Z}} A_k \twoheadrightarrow N$, that is, $N$ is $1$-periodic $(\mcA,\mcB)$-Gorenstein projective. 
\end{proof}

From the previous proof, one can show the following result which was inspired by \cite[Prop. 1.4]{BCE}.

\begin{proposition}
Let $(\mcA, \mcB)$ be a pair of classes of objects in an AB4 abelian category $\mcC$, with $\mcA$ closed under coproducts and direct summands. The following are equivalent:
\begin{enumerate}[(a)]
\item $\mathcal{GP}_{(\mcA, \mcB)} = \mcA$.

\item $\pi\mathcal{GP}_{(\mcA, \mcB, 1)} = \mcA$.
\end{enumerate}
\end{proposition}

\begin{proof}
Suppose first that $\mathcal{GP}_{(\mcA, \mcB)} = \mcA$ holds. From Remark~\ref{rem:sGP} (4) we have that
\[
\mcA \subseteq \pi\mathcal{GP}_{(\mcA, \mcB, 1)} \subseteq \mathcal{GP}_{(\mcA, \mcB)} = \mcA.
\]
Therefore, $\pi\mathcal{GP}_{(\mcA, \mcB, 1)} = \mcA$ holds true.

Now suppose that $\pi\mathcal{GP}_{(\mcA, \mcB, 1)} = \mcA$ holds, and let $C \in \mathcal{GP}_{(\mcA,\mcB)}$. As in Proposition~\ref{pro: GP sumandos directos de mGP}, we can get a $1$-periodic $(\mcA,\mcB)$-Gorenstein projective $N$ such that $C$ is a direct summand of $N$. Thus, $N\in \mcA$ since $\pi\mathcal{GP}_{(\mcA, \mcB, 1)}=\mcA$. Hence, $C\in \mcA$ since $\mcA$ is closed under direct summands.
\end{proof}

The following results show that objects in $\mcA$, $(\mcA,\mcB)$-Gorenstein objects and (proper) $m$-periodic $(\mcA,\mcB)$-Gorenstein objects are the same in the testing class $\mcB$, provided that $(\mcA,\mcB)$ is a GP-admissible or a right $n$-cotorsion pair.

\begin{proposition}\label{pro:SGPex=AcapB}
Let $(\mcA,\mcB)$ be a GP-admissible pair in $\mcC$ with $\omega:=\mcA\cap\mcB$ closed under direct summands. The following equalities hold: 
\[
\omega = \pi\mathcal{GP}_{(\mcA,\mcB,m)}^{\rm ppr} \cap \mcB = \pi\mathcal{GP}_{(\mcA,\mcB,m)}
\cap \mcB = \mathcal{GP}_{(\mcA,\mcB)}\cap \mcB.
\]
\end{proposition}

\begin{figure}[H]
\centering
\includegraphics[width=0.75\textwidth]{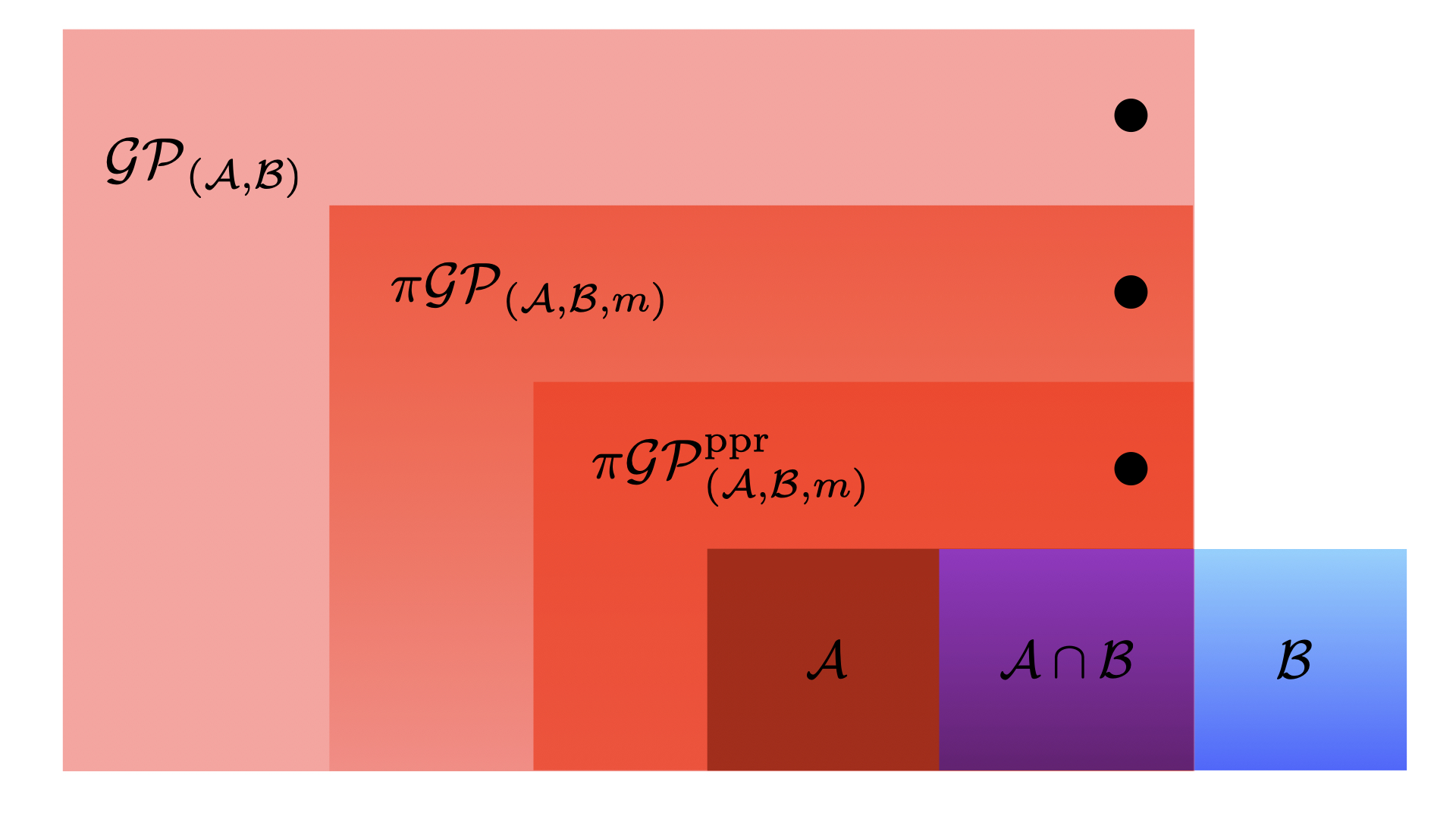}
\caption{Relations between the classes of objects in $\mcA$, $(\mcA,\mcB)$-Gorenstein projective objects, (proper) $m$-periodic $(\mcA,\mcB)$-Gorenstein projective objects, and their intersections with $\mcB$. The dots represent that the containments are strict, as shown in Remarks \ref{rem:sGP}-(4), \ref{rem:acy} and Example \ref{ex: piGPm->piGP1}.}
\end{figure}

\begin{proof}
Notice that 
\[
\omega \subseteq \pi\mathcal{GP}_{(\mcA,\mcB,m)}^{\rm ppr} \cap \mcB
\subseteq \pi\mathcal{GP}_{(\mcA,\mcB,m)} \cap \mcB \subseteq \mathcal{GP}_{(\mcA,\mcB)}\cap \mcB = \omega,
\] 
where the last equality follows from \cite[Coroll. 3.25 \& Thm. 3.32]{BMS}.
\end{proof}

\begin{corollary} 
Let $\mcC$ be an abelian category with enough projective objects, and $\mcB \subseteq \mcC$ be  closed under direct summands and finite coproducts. The following equalities hold:
\[
\mcP \cap \mcB = \pi\mathcal{GP}_{(\mcP,\mcB,m)}^{\rm ppr} \cap \mcB = \mathcal{GP}_{(\mathcal{P},\mcB)}\cap \mcB.
\]
\end{corollary}

\begin{proof} 
$(\mcP,\mcB)$ is GP-admissible with $\mcP \cap \mcB$ closed under direct summands. 
\end{proof}

Recall from \cite[Def. 2.1]{HMP} that $(\mcA, \mcB)$ is said to be \emph{right $n$-cotorsion pair}, with $n\geq 1$, if the following conditions hold:
\begin{enumerate}
\item $\mcB$ is closed under direct summands;
\item $\Ext^{\leq n}(\mcA, \mcB)=0$;
\item For every object $M\in \mcC$, there exists a short exact sequence $M\rightarrowtail B\twoheadrightarrow C$ with $B\in \mcB$ and $C\in \mcA^{\vee}_{n-1}$.
\end{enumerate}

\begin{proposition}\label{ejCT} 
Let $(\mcA,\mcB)$ be a right $n$-cotorsion pair where $\mcA$ is $(n+1)$-rigid. Then, 
\[
\pi\mathcal{GP}_{(\mcA,\mcB,m)}^{\rm ppr} = \mcB\cap \pi\mathcal{GP}_{(\mcA,\mcB,m)}.
\]
\end{proposition}

\begin{proof} 
Since $(\mcA,\mcB)$ is a right $n$-cotorsion pair, it follows from \cite[Dual of Thm. 2.7]{HMP} that $\mcB = \bigcap_{j=1}^n \mcA^{\perp_{j}}.$ Hence, by Proposition~\ref{prop:ex} we have
\[
\mcB\cap \pi\mathcal{GP}_{(\mcA,\mcB,m)} = \Big(\bigcap_{j=1}^n \mcA^{\perp_{j}} \Big) \cap \pi\mathcal{GP}_{(\mcA,\mcB,m)} = \pi\mathcal{GP}_{(\mcA,\mcB,m)}^{\rm ppr}.
\]
\end{proof}

\begin{example}\label{ex: ncot y acy}
Let $R$ be an $n$-Iwanaga-Gorenstein ring (where $n \geq 1$), that is, $R$ is two-sided noetherian with $\id({}_R R) = \id(R_R) = n$. We know from \cite[Ex. 5.1]{HMP} that $(\mathcal{I}(R), \mathcal{GI}(R))$ is a right $n$-cotorsion pair in $\mathrm{Mod}(R)$. Then,
\[
\pi\mathcal{GP}_{(\mathcal{I}(R), \mathcal{GI}(R), m)}^{\rm ppr}=\mathcal{GI}(R)\cap \pi\mathcal{GP}_{
(\mathcal{I}(R), \mathcal{GI}(R), m)}.
\]
\end{example}

The following is a consequence of Propositions~\ref{pro:SGPex=AcapB} and~\ref{ejCT}.

\begin{corollary} 
Let $(\mcA,\mcB)$ be a GP-admissible pair in $\mcC$ with $\omega := \mcA \cap \mcB$ closed under direct summands. If $(\mcA,\mcB)$ is a right $n$-cotorsion pair and $\mcA$ is $(n+1)$-rigid, then 
\[
\pi\mathcal{GP}_{(\mcA,\mcB,m)}^{\rm ppr} = \omega, 
\]
for every $1\leq m\leq n$.
\end{corollary}


\section{Applications to relative finitistic and \\ global Gorenstein dimensions}\label{sec:dimensions}

As mentioned in the introduction, one of the main applications of \cite[Thm. 2.7]{BM2} by Bennis and Mahdou was to show in \cite[Thm. 1.1]{BMglobal} the equality
\begin{align}
\sup\{\Gpd_{R}(M) : M\in \Mod(R)\} & = \sup\{\Gid_{R}(M) : M\in \Mod(R)\}. \label{eqn:BennisMahdou_global}
\end{align}
for any associative ring $R$ with identity. In other words, the global Gorenstein projective and Gorenstein injective dimensions of any ring coincide. On the other hand, global Gorenstein dimensions relative to GP-admissible pairs were studied by Becerril in \cite{Becerril}. So a natural question is wether it is possible to extend equality \eqref{eqn:BennisMahdou_global} for global Gorenstein projective and Gorenstein injective dimensions relative to GP-admissible and GI-admissible pairs. Partially, we answer this in the positive, under certain conditions for such pairs. 

In what follows, given an abelian category $\mcC$, $(\mathcal{A,B})$ will be a GP-admissible pair and $(\mathcal{Z,W})$ will be a GI-admissible pair in $\mcC$. We shall write $\omega := \mcA \cap \mcB$ and $\nu := \mathcal{Z} \cap \mathcal{W}$ for simplicity.  We shall prove several results concerning relative Gorenstein dimensions within the following settings: \\
\begin{itemize}
\item \textbf{Setup 1:}
\begin{enumerate}
\item $\omega$ and $\nu$ are closed under direct summands.

\item $\Ext^{1}(\pi\mathcal{GP}_{(\omega,\mcB,1)}, \nu) = 0$ and $\Ext^1(\omega,\pi\mathcal{GI}_{(\mathcal{Z},\nu,1)}) = 0$.

\item Every object in $\pi\mathcal{GP}_{(\omega, \mcB, 1)}$ admits a $\Hom(-,\nu)$-acyclic $\nu$-coresolution, and every object in $\pi\mathcal{GI}_{(\mathcal{Z}, \nu, 1)}$ admits a $\Hom(\omega,-)$-acyclic $\omega$-resolution. 

\item Every object in $\mathcal{GP}_{(\mathcal{A,B})}$ admits a $\omega$-resolution and a $\Hom(-,\nu)$-acyclic $\nu$-coresolution, and every object in $\mathcal{GI}_{(\mathcal{Z,W})}$ admits a $\nu$-coresolution and a $\Hom(\omega,-)$-acyclic $\omega$-resolution. 

\item $\nu$ (resp., $\omega$) is closed under arbitrary (co)products, in the case $\mcC$ is AB4${}^\ast$ (resp., AB4). \\
\end{enumerate}

\item \textbf{Setup 2:} 
\begin{enumerate}
\item $\omega$ and $\nu$ are closed under direct summands.

\item $\Ext^{1}(\mathcal{GP}_{(\mcA,\mcB)}, \nu) = 0$ and $\Ext^1(\omega,\mathcal{GI}_{(\mathcal{Z,W})}) = 0$.

\item Every object in $\mathcal{GP}_{(\mathcal{A,B})}$ admits a $\Hom(-,\nu)$-acyclic $\nu$-coresolution, and every object in $\mathcal{GI}_{(\mathcal{Z,W})}$ admits a $\Hom(\omega,-)$-acyclic $\omega$-resolution. \\
\end{enumerate}
\end{itemize}

Given $C \in \mcC$, the \emph{$(\mcA,\mcB)$-Gorenstein projective dimension of $C$} \cite[Def. 3.3]{BMS}, denoted ${\rm Gpd}_{(\mcA,\mcB)}(C)$, is defined as the resolution dimension 
\[
{\rm Gpd}_{(\mcA,\mcB)}(C) = \resdim_{\mathcal{GP}_{(\mcA,\mcB)}}(C).
\]
Similarly, 
\[
{\rm Gid}_{(\mathcal{Z,W})}(C) = \coresdim_{\mathcal{GI}_{(\mathcal{Z,W})}}(C)
\]
denotes and defines the \emph{$(\mathcal{Z,W})$-Gorenstein injective dimension of $C$}. If $\mathcal{X}$ is a class of objects in $\mcC$, the \emph{$(\mcA,\mcB)$-Gorenstein projective dimension} and \emph{$(\mathcal{Z,W})$-Gorenstein injective dimension of $\mathcal{X}$} are defined as
\begin{align*}
{\rm Gpd}_{(\mcA,\mcB)}(\mathcal{X}) & := \resdim_{\mathcal{GP}_{(\mcA,\mcB)}}(\mathcal{X}), \\
{\rm Gid}_{(\mathcal{Z,W})}(\mathcal{X}) & := \coresdim_{\mathcal{GI}_{(\mathcal{Z,W})}}(\mathcal{X}).
\end{align*}

\begin{lemma}\label{lem: Gid(piXcapY)<=n} 
The following assertions hold:
\begin{enumerate}
\item If conditions (1), (2) and (3) in Setup 1 are satisfied, then 
\begin{align*}
\Gid_{(\mathcal{Z, W})}(\pi\mathcal{GP}_{(\omega, \mcB, 1)}) & \leq \Gid_{(\mathcal{Z, W})}(\omega), \\ 
\Gpd_{(\mathcal{A,B})}(\pi\mathcal{GI}_{(\mathcal{Z}, \nu, 1)}) & \leq \Gpd_{(\mathcal{A,B})}(\nu).
\end{align*}

\item Under Setup 2., 
\begin{align*}
\Gid_{(\mathcal{Z, W})}(\mathcal{GP}_{(\mcA, \mcB)}) & \leq \Gid_{(\mathcal{Z, W})}(\omega), \\ \Gpd_{(\mathcal{A,B})}(\mathcal{GI}_{(\mathcal{Z}, \mathcal{W})}) & \leq \Gpd_{(\mathcal{A,B})}(\nu).
\end{align*}
\end{enumerate}
\end{lemma}

\begin{proof} \
\begin{enumerate}
\item Let $M \in \pi\mathcal{GP}_{(\omega, \mcB, 1)}$, and without loss of generality let $n := \Gid_{(\mathcal{Z, W})}(\omega) < \infty$. First, there exists an exact $\Hom(-,\mcB)$-acyclic sequence $M \rightarrowtail W \twoheadrightarrow M$ with $W \in \omega$. This sequence is also $\Hom(-, \nu)$-acyclic by condition (2) in Setup 1. Now, from condition (3) and the dual of \cite[Lem. 2.4]{BMPbalanced} we can construct the following solid diagram
\[
\parbox{2.5in}{
\begin{tikzpicture}[description/.style={fill=white,inner sep=2pt}] 
\matrix (m) [ampersand replacement=\&, matrix of math nodes, row sep=2.5em, column sep=2.5em, text height=1.25ex, text depth=0.25ex] 
{ 
M \& W \& M \\ 
V^0 \& V^0 \oplus V^0 \& V^0 \\
\vdots \& \vdots \& \vdots \\
V^{n-1} \& V^{n-1} \oplus V^{n-1} \& V^{n-1} \\
E^n \& E^n \oplus E^n \& E^n \\
}; 
\path[->] 
(m-2-1) edge (m-3-1) (m-3-1) edge (m-4-1)
; 
\path[>->]
(m-1-1) edge (m-1-2) edge (m-2-1)
(m-1-2) edge (m-2-2)
(m-1-3) edge (m-2-3)
(m-2-1) edge (m-2-2)
(m-4-1) edge (m-4-2)
(m-5-1) edge (m-5-2)
(m-2-2) edge (m-3-2)
(m-3-2) edge (m-4-2)
(m-2-3) edge (m-3-3)
(m-3-3) edge (m-4-3)
;
\path[->>]
(m-4-1) edge (m-5-1) (m-4-2) edge (m-5-2) (m-4-3) edge (m-5-3)
(m-5-2) edge (m-5-3)
(m-1-2) edge (m-1-3)
(m-2-2) edge (m-2-3)
(m-4-2) edge (m-4-3)
;
\end{tikzpicture} 
}
\]
where $V^k \in \nu$ for every $0 \leq k \leq n-1$. Notice that $V^j \oplus V^j \in \nu \subseteq \mathcal{GI}_{(\mathcal{Z, W})}$ since $\nu$ is closed under finite coproducts. Thus, by the dual of \cite[Coroll. 4.10]{BMS} and condition (1) we get that $E^n \oplus E^n \in \mathcal{GI}_{(\mathcal{Z, W})}$, and then $E_n \in \mathcal{GI}_{(\mathcal{Z, W})}$ by the dual of \cite[Coroll. 3.33]{BMS}. The left-hand side column yields $\Gid_{(\mathcal{Z,W})}(M) \leq n$. 

\item It suffices to notice that for every $M \in \mathcal{GP}_{(\mcA,\mcB)}$ we can consider a $\Hom(-,\mcB)$-acyclic exact sequence $M \rightarrowtail W \twoheadrightarrow M'$ with $W \in \omega$ and $M'\in \mathcal{GP}_{(\mathcal{A,B})}$ \cite[Prop. 3.14 \& Coroll. 3.25]{BMS}. The rest follows as in part (1).
\end{enumerate}
\end{proof}

The following result is another version of Proposition \ref{pro: GP sumandos directos de mGP}.

\begin{proposition}\label{prop: direct summd of pi1GP}
Let $\mcC$ be an AB4 and AB4${}^\ast$ abelian category. If conditions (4) and (5) in Setup 1 hold, then the following equalities hold:
\begin{align*}
\mathcal{GP}_{(\mathcal{A,B})} & = \mathcal{GP}_{(\omega,\mcB)} = {\rm add}(\pi\mathcal{GP}_{(\omega,\mcB,1)}), \\
\mathcal{GI}_{(\mathcal{Z,W})} & = \mathcal{GI}_{(\mathcal{Z},\nu)} = {\rm add}(\pi\mathcal{GI}_{(\mathcal{Z},\nu,1)}).
\end{align*}
\end{proposition}

\begin{proof}
From \cite[Thm. 3.32]{BMS} we have that the equality $\mathcal{WGP}_{(\omega, \mcB)}=\mathcal{GP}_{(\mcA, \mcB)}$. On the other hand, from the definition of the class $\mathcal{WGP}_{(\omega, \mcB)}$ and the assumption one can find for every $M \in \mathcal{WGP}_{(\omega, \mcB)}$ an exact complex
\[
W_{\bullet} \colon \cdots \to W_1 \to W_0 \to W_{-1} \to \cdots
\]
with $W_k \in \omega$ and $Z_k(W_\bullet) \in {}^{\perp}\mcB$ for every $k \in \mathbb{Z}$, and such that $M \simeq Z_0(W_{\bullet})$. Note that $Z_k(W_\bullet) \in {}^{\perp}\mcB$ holds for every $k > 0$ since $\omega \subseteq \mathcal{WGP}_{(\omega, \mcB)} = \mathcal{GP}_{(\mcA, \mcB)}$, $\mathcal{GP}_{(\mcA, \mcB)}$ is left thick and $\mathcal{GP}_{(\mcA, \mcB)} \subseteq {}^{\perp}\mcB$ (see \cite[Corolls. 3.15 \& 3.33]{BMS}). The rest of the proof follows as in the ``only if'' part of Proposition \ref{pro: GP sumandos directos de mGP}. 
\end{proof}

\begin{proposition}\label{pro: Gid(GP)<=n}
Let $\mcC$ be an AB4 and AB4${}^\ast$ abelian category. Under Setup 1, the following inequalities hold:
\begin{align*}
\Gid_{(\mathcal{Z, W})}(\mathcal{GP}_{(\mcA, \mcB)}) & \leq \Gid_{(\mathcal{Z, W})}(\omega), \\ 
\Gpd_{(\mathcal{A,B})}(\mathcal{GI}_{(\mathcal{Z},\mathcal{W})}) & \leq \Gpd_{(\mathcal{A,B})}(\nu).
\end{align*}
\end{proposition}

\begin{proof}
Without loss of generality, we may let $n := \Gid_{(\mathcal{Z, W})}(\omega) < \infty$. Now given $M\in \mathcal{GP}_{(\mcA, \mcB)}$, by Proposition~\ref{prop: direct summd of pi1GP} there exists $M' \in \mcC$ such that $M \oplus M' \in \pi\mathcal{GP}_{(\omega, \mcB, 1)}$. Note also that $M' \in {\rm add}(\pi\mathcal{GP}_{(\omega, \mcB, 1)}) = \mathcal{GP}_{(\mcA,\mcB)}$. On the other hand, by Lemma~\ref{lem: Gid(piXcapY)<=n} we have $\Gid_{(\mathcal{Z, W})}(M \oplus M') \leq n$. Now, since every object in $\mathcal{GP}_{(\mcA,\mcB)}$ admits a $\Hom(-,\nu)$-acyclic $\nu$-coresolution, we can form two exact sequences
\begin{align*}
M & \rightarrowtail V^0 \to \cdots \to V^{n-1} \twoheadrightarrow E^n, \\
M' & \rightarrowtail (V^0)' \to \cdots \to (V^{n-1})' \twoheadrightarrow (E^n)',
\end{align*}
with $V^j, (V^j)' \in \nu$ for every $0 \leq j \leq n-1$, from which we get the exact sequence 
\[
M \oplus M' \rightarrowtail V^0 \oplus (V^0)' \to \cdots \to V^{n-1} \oplus (V^{n-1})' \twoheadrightarrow E^n \oplus (E^n)'
\]
with $V^j \oplus (V^j)' \in \nu \subseteq \mathcal{GI}_{(\mathcal{Z,W})}$. Since $\Gid_{(\mathcal{Z, W})}(M\oplus M')\leq n$, it follows that $E^n \oplus (E^n)' \in \mathcal{GI}_{(\mathcal{Z, W})}$, and so $E^n \in \mathcal{GI}_{(\mathcal{Z, W})}$. Therefore, $\Gid_{(\mathcal{Z,W})}(M) \leq n$. The other inequality follows in a similar way.
\end{proof}

The following result extends Proposition~\ref{pro: Gid(GP)<=n} and describes the relation between the classes $\mathcal{GP}^\wedge_{(\mcA,\mcB)}$ and $\mathcal{GI}^\vee_{(\mathcal{Z,W})}$. Before presenting it, let us point out the following fact regarding $\mathcal{GI}_{(\mathcal{Z, W})}$ and $\mathcal{GP}_{(\mathcal{A,B})}$.

\begin{remark}\label{rmk: GI thick}
From \cite[Coroll. 4.10]{BMS} we know that $(\mathcal{GP}_{(\mcA,\mcB)},\omega)$ is a left Frobenius pair for any GP-admissible pair $(\mcA,\mcB)$ with $\omega$ closed under direct summands. Then, it follows by \cite[Thm. 2.11]{BMPS} that $\mathcal{GP}_{(\mcA,\mcB)}^\wedge$ is thick. Dually, one has that the class $\mathcal{GI}_{(\mathcal{Z,W})}^{\vee}$ is also thick for any GI-admissible pair $(\mathcal{Z,W})$ with $\nu$ closed under direct summands.
\end{remark}

\begin{proposition}\label{pro: GidGP^<=n}
If either $\mcC$ is an AB4 and AB4${}^\ast$ abelian category and the conditions of Setup 1 are satisfied, or $\mcC$ is an abelian category and the conditions of Setup 2 hold. Then, 
\begin{align*}
\Gid_{(\mathcal{Z, W})}(\mathcal{GP}_{(\mcA,\mcB)}^{\wedge}) & \leq \Gid_{(\mathcal{Z,W})}(\omega), \\
\Gpd_{(\mathcal{A,B})}(\mathcal{GI}_{(\mathcal{Z},\mathcal{W})}^{\vee}) & \leq \Gpd_{(\mathcal{A,B})}(\nu).
\end{align*}
\end{proposition}

\begin{proof}
Let $n := \Gid_{(\mathcal{Z, W})}(\omega) < \infty$ and $M \in \mathcal{GP}_{(\mcA,\mcB)}^{\wedge}$. We prove the first inequality by induction on $k := \Gpd_{(\mathcal{X, Y})}(M)$. 
\begin{itemize}
\item If $k = 0$ the result follows by Proposition~\ref{pro: Gid(GP)<=n}. 

\item For the case $k \geq 1$, we take a short exact sequence $G_1 \rightarrowtail G_0 \twoheadrightarrow M$ with $G_0 \in \mathcal{GP}_{(\mathcal{A,B})}$ and $\Gpd_{(\mathcal{A,B})}(G_1) = k - 1$. By the dual of \cite[Coroll. 4.11]{BMS} and induction hypothesis, we get that $\id_{\nu}(G_j) = \Gid_{(\mathcal{Z,W})}(G_j) \leq n$ for $j = 0, 1$, and hence $\Gid_{(\mathcal{Z, W})}(M) = \id_{\nu}(M) \leq n$ by Remark \ref{rmk: GI thick}. 
\end{itemize}
\end{proof}

We can use the previous result to show that the finitistic relative Gorenstein projective and injective dimensions of $\mcC$ are equal under Setups 1 and 2. These dimensions were defined in \cite[Def. 4.16]{BMS} as
\begin{align*}
{\rm FGPD}_{(\mathcal{A,B})}(\mcC) & := {\rm resdim}_{\mathcal{GP}_{(\mcA,\mcB)}}(\mathcal{GP}^\wedge_{(\mcA,\mcB)}), \\
{\rm FGID}_{(\mathcal{Z,W})}(\mcC) & := {\rm coresdim}_{\mathcal{GI}_{(\mathcal{Z,W})}}(\mathcal{GI}^\vee_{(\mathcal{Z,W})}).
\end{align*}
The following is a consequence of Proposition~\ref{pro: GidGP^<=n} and \cite[Coroll. 4.15 (b) and its dual]{BMS}.

\begin{theorem}\label{thm: gpd leq gid}
Suppose either $\mcC$ is an AB4 and AB4${}^\ast$ abelian category and the conditions of Setup 1 are satisfied, or $\mcC$ is an abelian category and the conditions of Setup 2 hold. 
If $\mathcal{GP}_{(\mathcal{A,B})}^{\wedge} = \mathcal{GI}_{(\mathcal{Z,W})}^{\vee}$ then 
\[
\rm{FGID}_{(\mathcal{Z,W})}(\mcC) = \Gid_{(\mathcal{Z,W})}(\omega) = \id_{\nu}(\omega) = \pd_{\omega}(\nu) = \Gpd_{(\mathcal{A,B})}(\nu) = \rm{FGPD}_{(\mathcal{A,B})}(\mcC).
\]
\end{theorem}

Related to the finitistic relative Gorenstein dimensions, we also have the global $(\mcA,\mcB)$-Gorenstein projective and global $(\mathcal{Z,W})$-Gorenstein injective dimensions, also defined in \cite[Def. 4.17]{BMS}, as
\begin{align*}
{\rm gl.GPD}_{(\mathcal{A,B})}(\mcC) & := {\rm Gpd}_{(\mcA,\mcB)}(\mathcal{C}), \\
{\rm gl.GID}_{(\mathcal{Z,W})}(\mcC) & := {\rm Gid}_{(\mathcal{Z,W})}(\mathcal{C}).
\end{align*}
One can note that 
\begin{align*}
{\rm FGPD}_{(\mathcal{A,B})}(\mcC) & \leq {\rm gl.GPD}_{(\mathcal{A,B})}(\mcC), \\
\rm{FGID}_{(\mathcal{Z,W})}(\mcC) & \leq {\rm gl.GID}_{(\mathcal{Z,W})}(\mcC).
\end{align*}
Note that $\mcC = \mathcal{GP}^\wedge_{(\mcA,\mcB)}$ implies that ${\rm FGPD}_{(\mathcal{A,B})}(\mcC) = {\rm gl.GPD}_{(\mathcal{A,B})}(\mcC)$, and a similar equality holds for the injective case if $\mathcal{C} = \mathcal{GI}^\vee_{(\mathcal{Z,W})}$. It follows that if $\mcC = \mathcal{GP}^\wedge_{(\mcA,\mcB)} = \mathcal{GI}^\vee_{(\mathcal{Z,W})}$, one obtains the equality
\begin{align}
{\rm gl.GID}_{(\mathcal{Z,W})}(\mcC) & = \rm{FGID}_{(\mathcal{Z,W})}(\mcC) = {\rm FGPD}_{(\mathcal{A,B})}(\mcC) = {\rm gl.GPD}_{(\mathcal{A,B})}(\mcC). \label{eqn:BennisMahdou_relative}
\end{align}
In particular, Theorem \ref{thm: gpd leq gid} covers \cite[Coroll. 3.5]{Becerril} in the case where $\mcC$ has enough projective and injective objects, and $\omega$ and $\nu$ coincide with the class of projective and injective objects of $\mcC$, respectively. The assumption $\mcC = \mathcal{GP}^\wedge_{(\mcA,\mcB)} = \mathcal{GI}^\vee_{(\mathcal{Z,W})}$ is sufficient but not necessary. For instance, in the absolute case $\mcA = \mcB =$ projective $R$-modules and $\mathcal{Z} = \mathcal{W} =$ injective $R$-modules, the equality \eqref{eqn:BennisMahdou_global} holds for any arbitrary ring $R$. In the relative setting, there are some known cases where \eqref{eqn:BennisMahdou_relative} holds, such as the case where $\mcC$ is an abelian category with enough projective and injective objects with ${\rm gl.GPD}_{(\mathcal{A,B})}(\mcC) < \infty$, $\omega =$ projective objects of $\mcC$, and ${\rm pd}(\mathcal{Z}) < \infty$ (see \cite[Prop. 3.15]{Becerril}). The following is a generalization of the cited result, which follows by Proposition \ref{pro: GidGP^<=n} and Theorem \ref{thm: gpd leq gid}.

\begin{corollary}\label{coro:relative_global}
Suppose either $\mcC$ is an AB4 and AB4${}^\ast$ abelian category and the conditions of Setup 1 are satisfied, or $\mcC$ is an abelian category and the conditions of Setup 2 hold. If either ${\rm gl.GPD}_{(\mathcal{A,B})}(\mcC) < \infty$ and $\Gid_{(\mathcal{Z,W})}(\omega) < \infty$, or ${\rm gl.GID}_{(\mathcal{Z,W})}(\mcC) < \infty$ and $\Gpd_{(\mathcal{A,B})}(\nu) < \infty$, then the following hold:
\[
{\rm gl.GID}_{(\mathcal{Z,W})}(\mcC) = \Gid_{(\mathcal{Z,W})}(\omega) = \pd_{\omega}(\nu) = \Gpd_{(\mathcal{A,B})}(\nu) = {\rm gl.GPD}_{(\mathcal{A,B})}(\mcC).
\]
\end{corollary}

During the rest of this section, we give some applications of Theorem~\ref{thm: gpd leq gid} in the context of modules over rings. Firstly, note that Bennis and Mahdou \cite[Thm. 1.1]{BMglobal} can be obtained from the previous corollary by setting $\mcA = \mcB = \mathcal{P}(R)$ (projective $R$-modules) and $\mathcal{Z} = \mathcal{W} = \mathcal{I}(R)$ (injective $R$-modules). Indeed, it is not hard to note that $(\mathcal{P}(R),\mathcal{P}(R))$ is a GP-admissible pair, and that $(\mathcal{I}(R),\mathcal{I}(R))$ is a GI-admissible pair. Moreover, conditions in Setup 1 and 2 are clearly satisfied for this choice of classes. Finally, if we assume that 
\[
\rm{gl.GPD}(R) := \Gpd_{(\mathcal{P}(R), \mathcal{P}(R))}(\Mod(R))
\]
is finite, then we know by \cite[Lem. 2.1]{BMglobal} that $\id(\mathcal{P}(R)) < \infty$. On the other hand, the containment $\mcI(R) \subseteq \mathcal{GI}(R)$ implies that ${\rm Gid}(\mathcal{P}(R)) \leq \id(\mathcal{P}(R)) < \infty$. A dual argument shows that ${\rm Gpd}(\mathcal{I}(R)) < \infty$ if we assume instead that 
\[
\rm{gl.GID}(R) := \Gid_{(\mathcal{I}(R), \mathcal{I}(R))}(\Mod(R))
\]
is finite. So the requirements in Corollary \ref{coro:relative_global} are met, and hence \eqref{eqn:BennisMahdou_global} holds.

Concerning Ding projective and Ding injective global dimensions, we know from the preliminaries and \cite{BMS} that $(\mathcal{P}(R), \mathcal{F}(R))$ is a GP-admissible pair and that $(\text{FP}\text{-}\mcI(R), \mathcal{I}(R))$ is a GI-admissible pair in $\Mod(R)$. Conditions in Setup 1 and 2 are clearly fulfilled for this choices. We can assert the same for GP-admissible and GI-admissible pairs $(\mathcal{P}(R), \mathcal{L}(R))$ and $(\text{FP}_\infty\text{-}\mcI(R), \mathcal{I}(R))$, respectively. We denote the corresponding global Ding projective (injective) and AC-Gorenstein projective (injective) dimensions as follows:
\begin{align*}
\rm{gl.DPD}(R) & := \Gpd_{(\mcP(R),\mathcal{F}(R))}(\Mod(R)), \\
\rm{gl.DID}(R) & := \Gid_{(\text{FP}\text{-}\mcI(R),\mcI(R))}(\Mod(R)), \\
\rm{gl.GPD}_{\rm{AC}}(R) & := \Gpd_{(\mcP(R),\mathcal{L}(R))}(\Mod(R)), \\
\rm{gl.GID}_{\rm{AC}}(R) & := \Gid_{(\text{FP}_\infty\text{-}\mcI(R),\mcI(R))}(\Mod(R)).
\end{align*}
The equality \eqref{eqn:BennisMahdou_global} can be extended to these dimensions after applying Corollary \ref{coro:relative_global}.

\begin{corollary}
The following equality holds for any arbitrary ring $R$:
\begin{align*}
{\rm gl.GPD}_{\rm{AC}}(R) & = \rm{gl.DPD}(R) = {\rm gl.GPD}(R) = {\rm gl.GID}(R) = \rm{gl.DID}(R) = {\rm gl.GID}_{\rm{AC}}(R).
\end{align*}
\end{corollary}

\begin{proof}
We show that if any of these six dimensions is finite, then so are the rest. First note that the containments $\mathcal{P}(R) \subseteq \mathcal{F}(R) \subseteq \mathcal{L}(R)$ and $\mathcal{I}(R) \subseteq \text{FP}\text{-}\mcI(R) \subseteq \text{FP}_\infty\text{-}\mcI(R)$ imply that 
\begin{align*}
{\rm gl.GPD}(R) & \leq \rm{gl.DPD}(R) \leq {\rm gl.GPD}_{\rm{AC}}(R), \\
{\rm gl.GID}(R) & \leq \rm{gl.DID}(R) \leq {\rm gl.GID}_{\rm{AC}}(R).
\end{align*}
Suppose for instance that $\rm{gl.GPD}(R) < \infty$ (the other cases are analogous), and consider the pairs $(\mathcal{P}(R),\mathcal{P}(R))$ and $(\text{FP}\text{-}\mcI(R),\mathcal{I}(R))$. We know that the finiteness of $\rm{gl.GPD}(R)$ implies that $\id(\mathcal{P}(R)) < \infty$. On the other hand, note that ${\rm Gid}_{(\text{FP}\text{-}\mcI(R),\mathcal{I}(R))}(\mathcal{P}(R)) \leq \id(\mathcal{P}(R))$. Thus, the hypotheses in Corollary \ref{coro:relative_global} hold, and so $\rm{gl.DID}(R) = {\rm gl.GPD}(R)$. Replacing the pair $(\text{FP}\text{-}\mcI(R),\mathcal{I}(R))$ by $(\mcI(R),\mathcal{I}(R))$ yields $\rm{gl.GID}(R) = {\rm gl.GPD}(R)$. In particular, $\rm{gl.GID}(R) < \infty$, and so a dual argument shows that $\rm{gl.DPD}(R) = {\rm gl.GID}(R)$. In a similar way, we can show that $\rm{gl.GID}_{\rm AC}(R) = {\rm gl.GPD}(R) = {\rm gl.GID}(R) = \rm{gl.GPD}_{\rm AC}(R)$
\end{proof}


\section*{Acknowledgements}

The authors want to thank professor Valente Santiago (Departamento de Mate-m\'aticas, Universidad Nacional Aut\'onoma de M\'exico), whose suggestions shortened some of the proofs appearing in an earlier version of this manuscript.


\section*{Funding}

The authors thank Project PAPIIT-Universidad Nacional Aut\'onoma de M\'exico IN103317 and IN100520. The first author was supported by Sociedad Matem\'atica Mexicana - Fundaci\'on Sof\'ia Kovaleskaia and Programa de Desarrollo de las Ciencias B\'asicas (PEDECIBA) postdoctoral fellowship. The third author was supported by the following grants and institutions: Fondo Vaz Ferreira \# II/FVF/2019/135 (funds are given by the Direcci\'on Nacional de Innovaci\'on, Ciencia y Tecnolog\'ia - Ministerio de Educaci\'on y Cultura, and administered through Fundaci\'on Julio Ricaldoni), Agencia Nacional de Investigaci\'on e Innovaci\'on (ANII), and PEDECIBA. 


\bibliographystyle{alpha}
\bibliography{biblionSG}

\end{document}